\documentclass[pdflatex,sn-mathphys-num]{sn-jnl}

\usepackage{graphicx}%
\usepackage{multirow}%
\usepackage{amsmath,amssymb,amsfonts}%
\usepackage{amsthm}%
\usepackage{mathrsfs}%
\usepackage[title]{appendix}%
\usepackage{xcolor}%
\usepackage{textcomp}%
\usepackage{manyfoot}%
\usepackage{booktabs}%
\usepackage{algorithm}%
\usepackage{algorithmicx}%
\usepackage{algpseudocode}%
\usepackage{listings}%

\theoremstyle{thmstyleone}%
\newtheorem{theorem}{Theorem}
\newtheorem{proposition}[theorem]{Proposition}%

\theoremstyle{thmstyletwo}%
\newtheorem{remark}{Remark}%

\newtheorem{lemma}[theorem]{Lemma}

\theoremstyle{thmstylethree}%
\newtheorem{definition}{Definition}%

\begin{document}

\title[Article Title]{Gauge Geometry of Hodge Zero-Mode Transport in Parameter-Dependent Topological Data Analysis}


\author*[1]{\fnm{Satoshi} \sur{Kanno}}\email{satoshi.kanno06@g.softbank.co.jp}

\author[1]{\fnm{Rei} \sur{Nishimura}}\email{rei.nishimura@g.softbank.co.jp}

\author[1]{\fnm{Hiroshi} \sur{Yamauchi}}\email{hiroshi.yamauchi@g.softbank.co.jp}

\author[1]{\fnm{Yoshi-aki} \sur{Shimada}}\email{yoshiaki.shimada01@g.softbank.co.jp}

\affil*[1]{
\orgdiv{Quantum Information Technology Department, Quantum Technology Division, Product Research and Development Division},
\orgname{SoftBank Corp.},
\orgaddress{
\street{1-7-1 Kaigan},
\city{Minato-ku},
\postcode{105-7529},
\state{Tokyo},
\country{Japan}
}
}


\abstract{We propose a practical computational framework for detecting structural changes in parameter-dependent topological data. In many applications, such as time-series data analysis, anomaly detection, and monitoring of systems under changing control parameters, persistence diagrams describe the birth and death of topological features at each parameter value, but they do not fully capture how these features are reorganized over time. To address this limitation, we represent homological features by zero modes of the ordinary combinatorial Hodge Laplacian and track the corresponding feature spaces in a common ambient chain space. This allows us to compute curvature and holonomy as descriptors of local reorganization and accumulated memory in evolving topological structures. Curvature highlights parameter regions where homological features mix or change rapidly, while holonomy summarizes the net effect of such changes after a closed cycle. We also establish stability estimates showing that these descriptors are robust under perturbations of the Hodge Laplacian on regular regions. Numerical experiments on controlled time-dependent point-cloud data show that the proposed method detects tracking instability, distinguishes systems with nearly identical persistence diagrams, and captures cycle-level memory invisible to pointwise feature matching. These results suggest that zero-mode transport geometry can serve as a useful computational tool for analyzing dynamic topological data.}




\maketitle

\section{Introduction}

Persistent homology \cite{edelsbrunner2002topological,zomorodian2004computing,cohen2005stability,edelsbrunner2008persistent,otter2017roadmap} is one of the central tools in topological data analysis. It extracts topological structures from data through a filtration and summarizes their robustness in terms of birth and death. Barcodes and persistence diagrams provide concise representations of the lifetimes of topological features, and have been widely used in both theory and applications, including glassy states\cite{obayashi2022persistent,nakamura2015persistent,hiraoka2016hierarchical}, materials science, machine learning, and neural network analysis\cite{moor2020topological,vandaele2021topologically,gabrielsson2020topology,hofer2020graph,adams2017persistence,hofer2017deep,bubenik2020persistence,chazal2014stochastic,kalivsnik2019tropical}.

In many applications, data depend not only on a filtration parameter but also on an additional external parameter, such as time, an external field, a control variable, an environmental variable, or an experimental condition. In such situations, it is practically important to describe and quantify how persistent homology, or more generally homological features, change as the external parameter varies. For example, when filtrations are constructed at successive times from time-series medical images, large changes in homological structure may be used as indicators of abnormal behavior. Similarly, when time-series data are analyzed by delay-coordinate embedding, or Takens embedding, changing the observation window produces a family of filtrations depending on a window parameter. In this case, detecting changes in the resulting homological features is useful for anomaly detection and the analysis of state transitions. Thus, tracking homology along variations of an external parameter is not merely a mathematical problem, but an important practical task in data analysis. In particular, one needs tools that describe not only the birth and death of topological features at each parameter value, but also how the corresponding homological feature spaces themselves vary over parameter space.

A representative viewpoint for such situations is given by vineyards \cite{cohen2006vines,chambers2026braiding}. Vineyards track how points in persistence diagrams move as the underlying data vary, thereby visualizing the evolution of topological features. This viewpoint is intuitive and has the advantage of directly displaying the motion of features in the diagram. By contrast, zigzag persistence \cite{carlsson2010zigzag} extends interval-decomposition-type descriptions to diagrams with both forward and backward maps, rather than aiming to describe the transport of feature spaces along an external parameter. Likewise, much of multiparameter persistence treats the additional parameter itself as a filtration coordinate, so that the entire structure becomes a genuine multifiltration, as in the rank invariant \cite{carlsson2007theory}, fibered barcodes and their visualization in RIVET \cite{lesnick2015interactive}, Hilbert functions \cite{oudot2024stability}, and multiparameter landscapes \cite{vipond2020multiparameter}. The setting of the present work is different: the scale parameter is regarded as the filtration parameter, while the second variable is treated as an external parameter such as time or a control variable.

Although vineyards provide an important description, any approach based on pointwise tracking has intrinsic limitations. When several significant persistence points approach one another, or when many short-lived features appear simultaneously, matching diagram points becomes unstable, and small perturbations can substantially change the correspondence. Moreover, even when local tracking is possible, it may fail to extend to a globally consistent labeling\cite{hickok2022persistence,hickok2022computing}. Thus, parameter-dependent topological structure cannot always be fully described by trajectories of labeled points in persistence diagrams.

This difficulty reflects the fact that homology is, at each parameter value, a finite-dimensional vector space. Choices of generators or representative cycles are generally not unique and need not be preserved under parameter variation. A basis that appears natural at one parameter value may become a different linear combination at another. Hence, tracking individual generators or feature labels is basis-dependent, whereas the intrinsic object is the variation of the homology space itself.

In this work, we therefore shift attention from individual persistence-diagram points to the reorganization of homological feature spaces. Rather than asking only how diagram points move, we ask how the corresponding homology spaces are transported, mixed, and reorganized over parameter space. Since homology is a finite-dimensional vector space at each parameter value, these spaces can be regarded as fibers over parameter space, and their variation can be formulated as a transport problem for homological vector spaces.

We make this transport geometry concrete using the ordinary combinatorial Hodge Laplacian. For a finite simplicial complex, the discrete Hodge theorem identifies the zero eigenspace of the Hodge Laplacian with homology. Thus, for a parameter-dependent filtration, the homology at each scale and external parameter value can be realized as a Hodge zero-mode space. The key point is that, instead of choosing individual generators or representative cycles, one can obtain a basis-independent description by using the orthogonal projection onto the zero-mode space.
More concretely, we embed the zero-mode spaces into a common Hilbert space and regard them as a zero-mode bundle over regular regions where the zero-mode multiplicity is constant and a spectral gap is maintained. On this bundle, a change of basis acts as a gauge transformation, while the curvature defined by the projection formula, as well as spectral quantities, invariant norms, and trace-type quantities derived from holonomy, provide gauge-invariant information. Therefore, our method describes the reorganization of homological feature spaces through zero-mode projections and gauge-invariant geometric quantities, rather than through generator tracking that depends on choices of bases or labels.

The proposed framework is not intended to replace vineyards. Rather, zero-mode transport is naturally connected to vineyard-type tracking. In regular regimes where persistence features are well separated and matching is stable, zero-mode transport is numerically consistent with vineyard-type monodromy. In this sense, the correspondences and exchanges described by vineyards can be viewed as coarse observations of transport in the zero-mode space. At the same time, because the proposed quantities are defined from the geometry of Hodge zero-mode subspaces rather than from a choice of labels or bases, transport, curvature, and holonomy remain stably computable even when persistence points approach one another and vineyard matching becomes ambiguous. This is one of the central numerical findings of this work.

The resulting quantities capture information that is difficult to detect by rank-type summaries or pointwise matching alone. Curvature measures local reorganization of homological feature spaces and the noncommutativity of infinitesimal transport, while holonomy records the cumulative effect of transport along closed loops, corresponding to global memory or monodromy. Thus, even when Betti numbers remain constant or persistence diagrams change only slightly, the internal structure of the zero-mode space may rotate, mix, or reorganize in a detectable way. The novelty of this work lies in formulating parameter-dependent homology as a transport theory of homological feature spaces and describing their reorganization through curvature and holonomy.

This viewpoint is also related to recent developments in quantum topological data analysis\cite{lloyd2016quantum,akhalwaya2024comparing,gyurik2022towards,yamauchi2025quantum,ameneyro2024quantum,hayakawa2022quantum,gyurik2024quantum}, where Hodge Laplacians, persistent Laplacians\cite{wei2025persistent,memoli2022persistent}, and other operator-based formulations play an important role. Since our construction is based directly on Laplacian zero modes and Berry-type geometry of subspaces, it suggests a possible connection between parameter-dependent topological data analysis and quantum geometric structures.

The main contributions of this work are as follows.

\begin{itemize}
    \item \textbf{A zero-mode transport geometry based on the ordinary Hodge Laplacian.}
    We formulate parameter-dependent homology as a transport geometry of zero-mode spaces associated with the ordinary combinatorial Hodge Laplacian. This provides a framework for describing the reorganization of homological feature spaces in terms of zero-mode projections, Berry-type connections, curvature, and holonomy.

    \item \textbf{Theoretical and numerical stability of the proposed zero-mode geometry.}
    We show that, on regular regions, the zero-mode projection and the curvature defined by the projection formula are locally stable under perturbations of the Hodge Laplacian as an operator family. Numerical experiments further support the stable computation of curvature under noise and small perturbations.

    \item \textbf{Numerical relation to vineyard tracking and detection of additional transport information.}
    We show numerically that zero-mode transport reproduces vineyard-type feature tracking in regular regimes, while remaining stably computable when persistence points approach one another and vineyard matching becomes unstable. We further show that curvature and holonomy detect mixing, reorganization, and cycle-level memory of homological feature spaces, which are difficult to capture by pointwise matching alone.
\end{itemize}

Overall, this work provides a framework for studying time-evolving data and data depending on external parameters through the transport geometry of homological feature spaces. It enables the analysis not only of the existence and lifetime of topological features, but also of the reorganization, transport, and memory effects of their internal structure.

This paper is organized as follows. Section 2 reviews parameter-dependent persistent homology and the limitations of pointwise feature tracking. Section 3 introduces the Hodge zero-mode bundle, together with curvature, holonomy, and stability on regular regions. Section 4 explains their interpretation as descriptors of local reorganization and global memory. Section 5 presents numerical experiments demonstrating the usefulness of the proposed descriptors. Section 6 concludes with discussion and future directions.

\section{Background and Motivation}

In this section, we review the basic concepts and notation used in this paper. We first recall the definitions of filtrations and persistent homology, and then formulate the setting in which an additional parameter independent of the filtration parameter is present. We then briefly discuss the idea of pointwise tracking in persistence diagrams, as represented by vineyards, and its limitations. Based on these preparations, in the next section we introduce the zero-mode bundle construction using the ordinary combinatorial Hodge Laplacian.

In the geometric and operator-theoretic parts of this paper, we use adjoint operators, orthogonal projections, Hodge Laplacians, and Berry connections. Therefore, chain groups are regarded as finite-dimensional Hilbert spaces over \(\mathbb R\) or \(\mathbb C\). Each chain group is equipped with the standard inner product for which the oriented simplices form an orthonormal basis. Hence the Hodge Laplacians and zero-mode projections used below are defined with respect to this inner product.

\subsection{Filtrations and persistent homology}

We first recall the basic notion of a filtration. Let \(I\) be a totally ordered set. A family of simplicial complexes
\[
\{K_a\}_{a\in I}
\]
is called a filtration if
\[
a\le b \quad \Longrightarrow \quad K_a\subset K_b
\]
for all \(a,b\in I\). In this paper, we mainly consider the case \(I\subset\mathbb R\). Typical examples include Vietoris--Rips filtrations and \v{C}ech filtrations associated with finite point clouds.

Given a filtration, for each \(a\le b\) there is an inclusion map
\[
\iota_{a,b}:K_a\hookrightarrow K_b.
\]
The image of the induced linear map on homology
\[
(\iota_{a,b})_*:H_q(K_a)\to H_q(K_b)
\]
is called the \(q\)-th persistent homology from \(a\) to \(b\):
\[
H_q^{a,b}
:=
\operatorname{Im}\left(H_q(K_a)\to H_q(K_b)\right).
\]
Its dimension
\[
\beta_q^{a,b}:=\dim H_q^{a,b}
\]
is called the persistent Betti number.

The persistent Betti number counts the number of \(q\)-dimensional features that have appeared by scale \(a\) and remain alive until scale \(b\). In this sense, persistent homology describes not only the existence of topological features, but also their robustness across scales.

In one-parameter persistent homology, every pointwise finite-dimensional persistence module decomposes as a direct sum of interval modules. Consequently, its isomorphism class is described by a barcode, equivalently by a persistence diagram. This is why birth--death data play a central role in one-parameter persistent homology.

\subsection{An additional parameter independent of the filtration}

In many applications, persistent homology depends not only on the filtration parameter, but also on an additional parameter independent of the filtration. The most typical example is time-dependent data, but the same situation also appears for data depending on external fields, control variables, temperature, pressure, environmental variables, or experimental conditions.

Let \(A\subset\mathbb R\) be the filtration-parameter set, and let \(\Lambda\) be an additional parameter space. We consider a family
\[
\{K(a,\lambda)\}_{(a,\lambda)\in A\times\Lambda}.
\]
If, for each fixed \(\lambda\in\Lambda\),
\[
a\le b \quad \Longrightarrow \quad K(a,\lambda)\subset K(b,\lambda),
\]
then this family is called a parameter-dependent filtration.

For each fixed \(\lambda\), one obtains the usual one-parameter persistent homology
\[
H_q^{a,b}(\lambda)
=
\operatorname{Im}
\left(
H_q(K(a,\lambda))\to H_q(K(b,\lambda))
\right).
\]
Thus, for each \(\lambda\), one may compute a persistence diagram or barcode. However, when an additional parameter is present, computing these objects separately for each \(\lambda\) does not by itself fully describe how the topological structure changes as \(\lambda\) varies.

In this paper, we do not regard the additional parameter as a second filtration coordinate. In other words, our setting is not that of genuine multiparameter persistence itself. Rather, we consider a situation in which an external parameter, such as time or a control variable, is present in addition to the filtration direction. From this viewpoint, the problem is not to construct a barcode-type complete invariant, but to describe how homological features vary along the external parameter.

\subsection{Vineyards and feature tracking}

A natural way to understand persistent homology depending on an additional parameter is to track points in the persistence diagram. This viewpoint is known as the vineyard construction. Informally, as \(\lambda\) varies, the persistence-diagram points
\[
(b_i(\lambda),d_i(\lambda))
\]
move in the plane, and their trajectories are interpreted as the evolution of topological features.

This viewpoint is natural and has the advantage of directly visualizing feature evolution. In particular, when the diagram points are well separated and the corresponding features are clearly distinguishable, vineyard-style tracking gives an effective description.

However, pointwise tracking has limitations. When several important features approach one another in the persistence diagram, it becomes unstable to decide which point should be regarded as the continuation of which. Moreover, when many short-lived features appear simultaneously, matching becomes sensitive to noise, and small perturbations may significantly change the resulting correspondence.

Furthermore, from the viewpoint of persistence diagram bundles, even if diagram points can be tracked locally, they may not glue together into a globally consistent labeling. Therefore, describing parameter-dependent topological structure solely in terms of trajectories of labeled persistence-diagram points has intrinsic limitations.

For this reason, in the following sections we introduce a framework for geometrically describing the variation of homological feature spaces using zero-mode spaces of the ordinary combinatorial Hodge Laplacian.

\section{Zero-Mode Bundles and Its Geometry}

In this section, we realize the family of zero-mode spaces of the ordinary combinatorial Hodge Laplacian as a family of subspaces inside a common Hilbert space, and show that it forms a vector bundle on a regular region. We then introduce a natural Berry-type connection on this zero-mode bundle and define the associated curvature and holonomy. Finally, we state that, on regular regions, the zero-mode projection and the curvature expressed by the projection formula are stable under perturbations of the operator family.

The Laplacian used in this section is not the persistent Laplacian, but the ordinary combinatorial Hodge Laplacian associated with the complex at each parameter point. Therefore, the construction in this section describes the transport geometry of instantaneous homology spaces at each scale and external parameter, rather than the transport geometry of fixed-birth persistent feature spaces.

\subsection{Zero-mode spaces and the zero-mode bundle}

We retain the notation of Section~2. Let \(A\subset \mathbb{R}\) be the set of
filtration parameters, and let \(\Lambda\) be an additional parameter space.


Throughout this subsection, we fix such a parameter-dependent filtration.

\begin{definition}[Ambient chain space]
Let \(K_{\max}\) denote the maximal simplicial complex on the fixed vertex set
\(V\). For each degree \(q\), define
\[
H_q:=C_q(K_{\max}).
\]
We regard \(H_q\) as a finite-dimensional Hilbert space over \(\mathbb{R}\) or
\(\mathbb{C}\), equipped with the standard inner product for which the oriented
\(q\)-simplices of \(K_{\max}\) form an orthonormal basis.
\end{definition}

For every \((d,\lambda)\in A\times\Lambda\), the chain group
\(C_q(K(d,\lambda))\) is naturally identified with a subspace of \(H_q\):
\[
C_q(K(d,\lambda))\subset H_q.
\]

\begin{definition}[Natural combinatorial Hodge Laplacian]
For each \((d,\lambda)\in A\times\Lambda\), let
\[
\partial_q(d,\lambda):
C_q(K(d,\lambda))\longrightarrow C_{q-1}(K(d,\lambda))
\]
be the boundary operator. The \emph{natural \(q\)-th combinatorial Hodge
Laplacian} on \(K(d,\lambda)\) is the self-adjoint operator
\[
\Delta_q^{\mathrm{Hodge},\mathrm{nat}}(d,\lambda)
:
C_q(K(d,\lambda))\longrightarrow C_q(K(d,\lambda))
\]
defined by
\[
\Delta_q^{\mathrm{Hodge},\mathrm{nat}}(d,\lambda)
=
\partial_q(d,\lambda)^\ast\partial_q(d,\lambda)
+
\partial_{q+1}(d,\lambda)\partial_{q+1}(d,\lambda)^\ast,
\]
where \(\ast\) denotes the adjoint with respect to the standard inner product.
\end{definition}

The operator above is defined on the parameter-dependent space
\(C_q(K(d,\lambda))\). In order to compare Laplacians at different parameter
values, we extend it to the fixed Hilbert space \(H_q\).

\begin{definition}[Extended Hodge Laplacian]
For each \((d,\lambda)\in A\times\Lambda\), consider the orthogonal
decomposition
\[
H_q
=
C_q(K(d,\lambda))\oplus C_q(K(d,\lambda))^\perp .
\]
The \emph{extended \(q\)-th Hodge Laplacian} is the self-adjoint operator
\[
\Delta_q^{\mathrm{Hodge}}(d,\lambda):H_q\longrightarrow H_q
\]
defined by
\[
\Delta_q^{\mathrm{Hodge}}(d,\lambda)
:=
\Delta_q^{\mathrm{Hodge},\mathrm{nat}}(d,\lambda)\oplus I,
\]
where \(I\) denotes the identity operator on
\(C_q(K(d,\lambda))^\perp\).
\end{definition}

\begin{remark}
The extension by the identity is chosen so that no additional zero modes are
created on the orthogonal complement. In particular,
\[
\ker \Delta_q^{\mathrm{Hodge}}(d,\lambda)
=
\ker \Delta_q^{\mathrm{Hodge},\mathrm{nat}}(d,\lambda).
\]
\end{remark}

\begin{definition}[Zero-mode space]
For each \((d,\lambda)\in A\times\Lambda\), the \emph{\(q\)-th zero-mode space}
of \(K(d,\lambda)\) is defined by
\[
E_0(d,\lambda)
:=
\ker \Delta_q^{\mathrm{Hodge}}(d,\lambda)
\subset H_q.
\]
\end{definition}

The following proposition shows that the zero-mode space recovers the usual
simplicial homology.

\begin{proposition}
For every \((d,\lambda)\in A\times\Lambda\), there is a natural isomorphism
\[
E_0(d,\lambda)\cong H_q(K(d,\lambda)).
\]
In particular,
\[
\dim E_0(d,\lambda)=\beta_q(d,\lambda).
\]
\end{proposition}

\begin{proof}
By definition of the extended Hodge Laplacian, we have
\[
\Delta_q^{\mathrm{Hodge}}(d,\lambda)
=
\Delta_q^{\mathrm{Hodge},\mathrm{nat}}(d,\lambda)\oplus I
\]
with respect to the orthogonal decomposition
\[
H_q
=
C_q(K(d,\lambda))\oplus C_q(K(d,\lambda))^\perp .
\]
Since the identity operator on \(C_q(K(d,\lambda))^\perp\) has trivial kernel,
it follows that
\[
\ker \Delta_q^{\mathrm{Hodge}}(d,\lambda)
=
\ker \Delta_q^{\mathrm{Hodge},\mathrm{nat}}(d,\lambda).
\]
On the other hand, by the discrete Hodge theorem,
\[
\ker \Delta_q^{\mathrm{Hodge},\mathrm{nat}}(d,\lambda)
\cong
H_q(K(d,\lambda)).
\]
Therefore,
\[
E_0(d,\lambda)
=
\ker \Delta_q^{\mathrm{Hodge}}(d,\lambda)
\cong
H_q(K(d,\lambda)).
\]
Taking dimensions gives
\[
\dim E_0(d,\lambda)=\beta_q(d,\lambda).
\]
This completes the proof.
\end{proof}

\subsection{Regular regions and smooth operator models}

In this subsection, we formulate the regions on which the zero-mode spaces of
parameter-dependent Hodge Laplacians can be treated as vector bundles. In
particular, we clarify the regularity and spectral-gap assumptions needed in
order to define Berry connections, curvature, and holonomy.

Vietoris--Rips and \v{C}ech filtrations do not, in general, vary smoothly with
respect to parameters. For example, consider a finite point cloud depending on a
parameter \(\lambda\),
\[
X(\lambda)=\{x_1(\lambda),\dots,x_N(\lambda)\}.
\]
The Vietoris--Rips complex at scale \(d\) is given by
\[
K(d,\lambda)
=
\{\sigma\subset X(\lambda)\mid \operatorname{diam}_\lambda(\sigma)\le d\}.
\]
Thus, a simplex \(\sigma\) enters the complex precisely when one crosses the
critical hypersurface
\[
d=\operatorname{diam}_\lambda(\sigma).
\]
For this reason, the ordinary combinatorial Hodge Laplacian is not, in general,
a globally smooth operator family in the parameters \((d,\lambda)\). Rather, it
is an operator family that changes discretely as simplices enter or leave the
complex.

In this work, we model this discretely varying operator by a smooth Hodge-type
operator family on a common Hilbert space. This does not mean that the filtration
itself is replaced by a smooth topological invariant. Rather, the entrance and
exit of simplices are modeled by smooth activation weights so that Berry
connections and curvature can be defined.

\begin{definition}[Filtration threshold]
For each simplex \(\sigma\in K_{\max}\), let
\[
r_\sigma(\lambda)
\]
denote the filtration threshold at which \(\sigma\) appears. In the
Vietoris--Rips case, this is typically given by
\[
r_\sigma(\lambda)=\operatorname{diam}_\lambda(\sigma).
\]
\end{definition}

In the ordinary filtration, the presence or absence of the simplex \(\sigma\) is
described by the discontinuous indicator
\[
\mathbf 1_{\{d\ge r_\sigma(\lambda)\}}.
\]

\begin{definition}[Smooth activation weight]
Let \(\rho_\varepsilon\) be a smooth approximation of the step function. 
For each simplex \(\sigma\in K_{\max}\), define the smooth activation weight
\[
w_\sigma^\varepsilon(d,\lambda)
=
\rho_\varepsilon(d-r_\sigma(\lambda)).
\]
\end{definition}

For example, one may take
\[
\rho_\varepsilon(s)
=
\frac{1}{1+e^{-s/\varepsilon}}.
\]
If \(d\ll r_\sigma(\lambda)\), one has
\[
w_\sigma^\varepsilon(d,\lambda)\approx 0,
\]
so that the simplex \(\sigma\) is almost inactive. On the other hand, if
\(d\gg r_\sigma(\lambda)\), then
\[
w_\sigma^\varepsilon(d,\lambda)\approx 1,
\]
so that the simplex \(\sigma\) is almost fully active. Only near the critical
hypersurface
\[
d=r_\sigma(\lambda)
\]
does the weight smoothly transition from \(0\) to \(1\).

\begin{remark}
The sigmoid function is used only as a simple smooth approximation of the
discontinuous simplex indicator. Since it is not exactly equal to \(0\) or \(1\)
for finite \(\varepsilon\), the resulting weighted operator is not, in general,
the exact simplicial Hodge Laplacian.

Nevertheless, with a sufficiently sharp activation profile, or with a smooth
cutoff function that is exactly \(0\) and \(1\) outside a small transition
region, the operator agrees with, or closely approximates, the ordinary Hodge
Laplian away from critical hypersurfaces.
\end{remark}

\begin{definition}[Weight operator]
For each degree \(q\), define the diagonal operator
\[
W_q^\varepsilon(d,\lambda)
=
\operatorname{diag}
\bigl(
w_\sigma^\varepsilon(d,\lambda)
\bigr)_{\sigma\in K_{\max}^{(q)}}.
\]
This is a self-adjoint operator on the common Hilbert space \(H_q\), and records
the activation strength of each \(q\)-simplex direction.
\end{definition}

\begin{definition}[Weighted boundary operator]
Let
\[
\partial_q:H_q\to H_{q-1}
\]
be the ordinary boundary operator on the maximal complex \(K_{\max}\). The
weighted boundary operator is defined by
\[
\partial_q^\varepsilon(d,\lambda)
=
\bigl(W_{q-1}^\varepsilon(d,\lambda)\bigr)^{1/2}
\partial_q
\bigl(W_q^\varepsilon(d,\lambda)\bigr)^{1/2}.
\]
\end{definition}

This definition smoothly turns the boundary relation on or off according to the
activation of the \(q\)-simplex on the domain side and the \((q-1)\)-simplex on
the target side. If both simplices are sufficiently active, the operator is close
to the ordinary boundary operator. If either simplex is inactive, its
contribution is suppressed.

\begin{definition}[Smooth Hodge-type operator]
Fix a parameter \(\mu>0\). The smooth Hodge-type operator in degree \(q\) is
defined by
\[
\Delta_q^\varepsilon(d,\lambda)
=
(\partial_q^\varepsilon)^\ast\partial_q^\varepsilon
+
\partial_{q+1}^\varepsilon(\partial_{q+1}^\varepsilon)^\ast
+
\mu\bigl(I-W_q^\varepsilon(d,\lambda)\bigr).
\]
Here \(\mu\) is a penalty parameter assigning positive energy to inactive
\(q\)-simplex directions.
\end{definition}

The final term
\[
\mu(I-W_q^\varepsilon)
\]
pushes inactive \(q\)-simplex directions away from the zero-mode space. Indeed,
in directions where \(w_\sigma^\varepsilon\approx 1\), the penalty almost
disappears, while in directions where \(w_\sigma^\varepsilon\approx 0\), a
positive energy of order \(\mu\) is assigned.

\begin{remark}[Choice of the penalty parameter]
The role of \(\mu\) is only to assign positive energy to inactive simplex
directions. Thus, from the theoretical point of view, any fixed \(\mu>0\) is
sufficient.

In the numerical experiments, unless otherwise stated, we take
\[
\mu=1.
\]
This is the natural choice that assigns a unit penalty to inactive simplex
directions and is sufficient to remove those directions from the zero-mode
sector. More generally, one may choose \(\mu\) according to the typical scale of
the nonzero spectrum, but in this work we use the fixed value \(\mu=1\).
\end{remark}

\begin{proposition}[Basic properties of the smooth Hodge-type operator]
The operator
\[
\Delta_q^\varepsilon(d,\lambda)
\]
defined above is self-adjoint and positive semidefinite on the common Hilbert
space \(H_q\).

Moreover, if each weight \(w_\sigma^\varepsilon(d,\lambda)\) is of class \(C^r\)
in \((d,\lambda)\), then
\[
(d,\lambda)\longmapsto \Delta_q^\varepsilon(d,\lambda)
\]
is a \(C^r\) family of self-adjoint operators.
\end{proposition}

\begin{proof}
First,
\[
(\partial_q^\varepsilon)^\ast\partial_q^\varepsilon
\]
is of the form \(A^\ast A\), with \(A=\partial_q^\varepsilon\). Hence it is
self-adjoint and positive semidefinite. Similarly,
\[
\partial_{q+1}^\varepsilon(\partial_{q+1}^\varepsilon)^\ast
\]
is of the form \(BB^\ast\), and is therefore self-adjoint and positive
semidefinite.

Furthermore, if
\[
0\le w_\sigma^\varepsilon(d,\lambda)\le 1
\]
for all \(\sigma\), then the diagonal operator \(W_q^\varepsilon(d,\lambda)\)
satisfies
\[
0\le W_q^\varepsilon(d,\lambda)\le I.
\]
Hence
\[
I-W_q^\varepsilon(d,\lambda)
\]
is self-adjoint and positive semidefinite. Since \(\mu>0\), the operator
\[
\mu(I-W_q^\varepsilon(d,\lambda))
\]
is also self-adjoint and positive semidefinite.

Therefore \(\Delta_q^\varepsilon(d,\lambda)\), being the sum of these three
terms, is self-adjoint and positive semidefinite.

If \(w_\sigma^\varepsilon(d,\lambda)\) depends on \((d,\lambda)\) in a \(C^r\)
manner, then so does \(W_q^\varepsilon(d,\lambda)\). Moreover, since
\(W_q^\varepsilon(d,\lambda)\) is diagonal with nonnegative entries, its square
root also depends \(C^r\)-smoothly on \((d,\lambda)\), under the chosen smooth
activation model. Consequently,
\[
\partial_q^\varepsilon(d,\lambda)
\]
is a \(C^r\) family of operators, and hence so is
\[
\Delta_q^\varepsilon(d,\lambda).
\]
\end{proof}

\begin{remark}
The smooth ambient Hodge model is not intended to define a new topological
invariant. Near critical hypersurfaces, simplices may have intermediate weights
between \(0\) and \(1\), and therefore
\[
\ker \Delta_q^\varepsilon(d,\lambda)
\]
need not coincide exactly with simplicial homology at every parameter value.

Rather, the zero-mode space, or low-energy space, of
\(\Delta_q^\varepsilon(d,\lambda)\) should be interpreted as the feature space of
a smooth Hodge-type operator model associated with the discrete filtration. On
the other hand, away from critical hypersurfaces and for sufficiently small
\(\varepsilon\), each \(w_\sigma^\varepsilon\) is close to either \(0\) or \(1\),
and the operator approximates the ordinary Hodge Laplacian together with a
positive penalty on inactive simplex directions.
\end{remark}

\begin{definition}[Regular region]
Let \(U\subset A\times\Lambda\) be an open set, and let
\[
\Delta_q(d,\lambda)
\]
be the chosen Hodge-type operator family in degree \(q\). We say that \(U\) is a
\emph{regular region} for \(\Delta_q\) if the following conditions hold:

\begin{enumerate}
    \item The map
    \[
    (d,\lambda)\longmapsto \Delta_q(d,\lambda)
    \]
    is continuous on \(U\), and is of class \(C^r\) whenever derivatives are
    needed.

    \item The zero-mode space
    \[
    E_0(d,\lambda):=\ker \Delta_q(d,\lambda)
    \]
    has constant dimension on \(U\).

    \item There exists a constant \(\gamma>0\) such that, for every
    \((d,\lambda)\in U\),
    \[
    \operatorname{Spec}\bigl(\Delta_q(d,\lambda)\bigr)
    \subset
    \{0\}\cup[\gamma,\infty).
    \]
\end{enumerate}
\end{definition}

Thus, on a regular region, the zero eigenvalue is separated from the nonzero
spectrum by a uniform spectral gap.

\begin{remark}[Spectral gap condition]
The spectral-gap condition ensures that the zero-mode space remains stably
separated from the nonzero modes.

If the smallest positive eigenvalue is denoted by
\[
\lambda_1^+(d,\lambda)
:=
\min\left(
\operatorname{Spec}(\Delta_q(d,\lambda))
\setminus\{0\}
\right),
\]
then on a regular region we require
\[
\lambda_1^+(d,\lambda)\ge \gamma
\]
uniformly.
\end{remark}

\begin{proposition}[Riesz projection onto the zero-mode space]
Let \(U\subset A\times\Lambda\) be a connected regular region. Then the orthogonal
projection
\[
P_0(d,\lambda):H_q\to H_q
\]
onto the zero-mode space is given by the Riesz projection
\[
P_0(d,\lambda)
=
\frac{1}{2\pi i}
\int_\Gamma
\bigl(z-\Delta_q(d,\lambda)\bigr)^{-1}\,dz.
\]
Here \(\Gamma\) is a small closed contour enclosing the origin and no nonzero
spectrum.

Moreover, if \(\Delta_q(d,\lambda)\) is a \(C^r\) operator family, then
\[
(d,\lambda)\longmapsto P_0(d,\lambda)
\]
is also of class \(C^r\).
\end{proposition}

\begin{proof}
By the definition of a regular region, the zero eigenvalue is separated from the
nonzero spectrum by a uniform positive distance. Therefore, one can choose a
closed contour \(\Gamma\), uniformly over \(U\), which encloses the origin and no
nonzero spectrum.

By the standard spectral decomposition for finite-dimensional operators, or
equivalently by the Riesz functional calculus,
\[
\frac{1}{2\pi i}
\int_\Gamma
\bigl(z-\Delta_q(d,\lambda)\bigr)^{-1}\,dz
\]
is the spectral projection onto the generalized eigenspace corresponding to the
eigenvalues enclosed by \(\Gamma\).

Since \(\Delta_q(d,\lambda)\) is self-adjoint and the only eigenvalue enclosed by
\(\Gamma\) is \(0\), this projection is precisely the orthogonal projection onto
\[
E_0(d,\lambda)=\ker \Delta_q(d,\lambda).
\]
This proves the formula.

If \(\Delta_q(d,\lambda)\) is of class \(C^r\), then for each \(z\in\Gamma\) the
resolvent
\[
\bigl(z-\Delta_q(d,\lambda)\bigr)^{-1}
\]
also depends \(C^r\)-smoothly on \((d,\lambda)\). Since the contour \(\Gamma\) is
fixed, integration over \(\Gamma\) preserves \(C^r\)-regularity. Hence
\[
P_0(d,\lambda)
\]
is a \(C^r\) family of projections.
\end{proof}

\begin{proposition}[Zero-mode bundle]
Let \(U\subset A\times\Lambda\) be a connected regular region. Then
\[
E_0
:=
\bigsqcup_{(d,\lambda)\in U}E_0(d,\lambda)
\]
forms a finite-dimensional vector bundle over \(U\). Its fiber at
\((d,\lambda)\in U\) is
\[
E_0(d,\lambda)=\operatorname{Im}P_0(d,\lambda).
\]
\end{proposition}

\begin{proof}
By the definition of a regular region,
\[
\dim E_0(d,\lambda)
\]
is constant on \(U\). Moreover, by the preceding proposition, the projection
\[
P_0(d,\lambda)
\]
depends continuously, or \(C^r\)-smoothly, on \((d,\lambda)\).

Therefore,
\[
\operatorname{Im}P_0(d,\lambda)
\]
defines a continuous, or \(C^r\), family of fixed-dimensional subspaces of the
fixed finite-dimensional Hilbert space \(H_q\).
Hence
\[
E_0
=
\bigsqcup_{(d,\lambda)\in U}E_0(d,\lambda)
\]
is a vector bundle over \(U\).
\end{proof}

\begin{definition}[Hodge zero-mode bundle]
Let \(U\subset A\times\Lambda\) be a connected regular region. The vector bundle
\[
E_0\to U
\]
constructed above is called the \emph{\(q\)-th Hodge zero-mode bundle}, or the
\emph{instantaneous zero-mode bundle}, over \(U\).
\end{definition}

\begin{definition}[Singular set and regular part]
Let \(\Sigma\subset A\times\Lambda\) be the set of points at which at least one
of the following occurs:
\begin{enumerate}
    \item \(\dim E_0(d,\lambda)\) is not locally constant;
    \item the spectral gap between the zero eigenvalue and the nonzero spectrum
    is lost.
\end{enumerate}
We call \(\Sigma\) the \emph{singular set}. Its complement
\[
(A\times\Lambda)^\circ
:=
(A\times\Lambda)\setminus\Sigma
\]
is called the \emph{regular part}.
\end{definition}

\begin{remark}
The geometric theory developed in this paper is not a globally smooth theory
over the entire parameter space. At critical hypersurfaces where the spectral gap closes, the smooth vector-bundle picture need not hold.
Such points are excluded as part of the singular set \(\Sigma\). The Berry
connections, curvature, and holonomy considered in this work are understood as
geometric quantities of the Hodge zero-mode bundle defined on regular regions
away from the singular set.

Strictly speaking, if the regular region \(U\) is not connected, the
zero-mode multiplicity may be constant only on each connected component and
may differ from one component to another. In that case, \(E_0\) should not be
regarded as a single vector bundle of fixed rank over all of \(U\), but rather as
a collection of vector bundles over the connected components of \(U\). In the
following discussion, we always work on one connected component on which
\(\dim E_0(d,\lambda)=m\) is fixed. For simplicity, we continue to call this object
the zero-mode bundle.
\end{remark}

\subsection{Natural connection and Berry connection}

Let \(U\subset A\times\Lambda\) be a regular region, and consider the zero-mode bundle
\[
E_0\to U.
\]
For each \((d,\lambda)\in U\), the fiber \(E_0(d,\lambda)\) is a subspace of the common Hilbert space \(H_q\), and \(P_0(d,\lambda)\) is the orthogonal projection onto it. Hence \(P_0\) induces a natural connection on \(E_0\).

\begin{definition}
For a local section \(s\in\Gamma(U,E_0)\), define
\[
\nabla s:=P_0(ds).
\]
\end{definition}

\begin{proposition}
The operator \(\nabla\) defines a connection on the vector bundle \(E_0\to U\).
\end{proposition}

\begin{proof}
Linearity of \(\nabla\) follows from the definition. Let \(f\in C^\infty(U)\) and \(s\in\Gamma(U,E_0)\). Then
\[
\nabla(fs)
=
P_0(d(fs))
=
P_0(df\otimes s+f\,ds).
\]
Since \(s\) is a section of \(E_0\), we have \(P_0s=s\) pointwise. Therefore
\[
P_0(df\otimes s)=df\otimes s,
\qquad
P_0(f\,ds)=fP_0(ds)=f\nabla s.
\]
Thus
\[
\nabla(fs)=df\otimes s+f\nabla s.
\]
This is the Leibniz rule, so \(\nabla\) is a connection.
\end{proof}

Let \(\operatorname{rank}E_0=m\), and choose a local orthonormal frame
\[
\psi_1(d,\lambda),\dots,\psi_m(d,\lambda)
\]
on \(U\). Arrange these vectors as columns into the matrix
\[
\Psi(d,\lambda)
:=
(\psi_1(d,\lambda),\dots,\psi_m(d,\lambda)).
\]
Then
\[
\Psi^\ast\Psi=I_m,
\qquad
P_0=\Psi\Psi^\ast.
\]

Every local section \(s\) can be written uniquely as
\[
s=\Psi\xi,
\]
where \(\xi:U\to\mathbb C^m\) is a coefficient vector. In the real case, \(\mathbb C^m\) is replaced by \(\mathbb R^m\). Then
\[
\nabla s
=
\Psi\bigl(d\xi+\Psi^\ast d\Psi\,\xi\bigr).
\]
Thus the connection \(1\)-form in the frame \(\Psi\) is
\[
A:=\Psi^\ast d\Psi.
\]

\begin{definition}
The matrix-valued \(1\)-form
\[
A=\Psi^\ast d\Psi
\]
is called the Berry connection on the zero-mode bundle.
\end{definition}

Differentiating \(\Psi^\ast\Psi=I_m\), we obtain
\[
d\Psi^\ast\,\Psi+\Psi^\ast d\Psi=0.
\]
Hence
\[
A^\ast=-A.
\]
In particular, in the complex case, \(A\) is a \(\mathfrak u(m)\)-valued \(1\)-form. In the real case, it is an \(\mathfrak o(m)\)-valued \(1\)-form.

If the local orthonormal frame is changed by
\[
\Psi'=\Psi g,
\]
then in the complex case \(g:U\to U(m)\), while in the real case \(g:U\to O(m)\). The corresponding connection \(1\)-form \(A'\) satisfies
\[
A'=g^{-1}Ag+g^{-1}dg.
\]
Thus the local expression of the Berry connection is gauge-dependent.

\subsection{Curvature and holonomy}

We now define the curvature and holonomy associated with the Berry connection.

\begin{definition}
The curvature of the Berry connection \(A\) is defined by
\[
F:=dA+A\wedge A.
\]
\end{definition}

If \(U\) is a two-dimensional parameter space with local coordinates \((d,t)\), then
\[
A=A_d\,dd+A_t\,dt,
\]
and
\[
F=F_{dt}\,dd\wedge dt,
\]
where
\[
F_{dt}
=
\partial_dA_t-\partial_tA_d+[A_d,A_t].
\]

The curvature can also be expressed solely in terms of the projection \(P_0\).

\begin{proposition}
The curvature satisfies
\[
F=P_0[dP_0,dP_0]P_0.
\]
\end{proposition}

\begin{proof}
This follows by a standard computation using \(P_0=\Psi\Psi^\ast\) and \(A=\Psi^\ast d\Psi\).
\end{proof}

This projection formula is particularly important in the present work. The connection form \(A\) depends on the choice of local frame, whereas the projection \(P_0\) is intrinsically defined as an operator on the common Hilbert space. Thus, by writing the curvature as
\[
F=P_0[dP_0,dP_0]P_0,
\]
we obtain a gauge-independent description of the noncommutativity of local zero-mode transport.

Next, let \(\gamma:[0,1]\to U\) be a piecewise \(C^1\) curve. The parallel-transport equation in a local frame is
\[
\frac{d}{ds}\xi(s)
+
A(\dot\gamma(s))\xi(s)
=
0.
\]
Its solution defines a linear map
\[
\Pi_\gamma:E_0(\gamma(0))\to E_0(\gamma(1)).
\]

\begin{definition}
If \(\gamma\) is a closed curve, that is, if \(\gamma(0)=\gamma(1)=x\), then
\[
U_\gamma:=\Pi_\gamma:E_0(x)\to E_0(x)
\]
is called the holonomy along \(\gamma\).
\end{definition}

In a local frame, the holonomy is written as
\[
U_\gamma
=
\mathcal{P}\exp\left(-\int_\gamma A\right),
\]
where \(\mathcal{P}\) denotes path ordering.

Under a gauge transformation \(\Psi'=\Psi g\), the matrix representation of the holonomy transforms by conjugation:
\[
U_\gamma'
=
g(\gamma(0))^{-1}U_\gamma g(\gamma(0)).
\]
Therefore,
\[
\operatorname{tr}(U_\gamma),
\qquad
\det(U_\gamma),
\qquad
\operatorname{Spec}(U_\gamma)
\]
are gauge-invariant.

\subsection{Stability on regular regions}

Finally, we state the local stability of the zero-mode projection and the curvature on regular regions. To obtain an intrinsic stability statement that does not depend on gauge choices, we do not directly compare local connection forms \(A\) or holonomy matrices. Instead, we focus on \(P_0\), which can be compared as an operator on the common Hilbert space, and on the curvature expressed by the projection formula
\[
F=P_0[dP_0,dP_0]P_0.
\]

Let \(U\subset A\times\Lambda\) be a regular region, and let
\[
L(d,\lambda),
\qquad
\widetilde L(d,\lambda)
\]
be two self-adjoint operator families on \(H_q\). Here we regard them as common-Hilbert-space realizations of ordinary Hodge Laplacians, or as perturbations of such realizations. Let
\[
P_0,\qquad \widetilde P_0
\]
denote the corresponding zero-mode projections. Define the corresponding curvatures by
\[
F:=P_0[dP_0,dP_0]P_0,
\qquad
\widetilde F:=
\widetilde P_0[d\widetilde P_0,d\widetilde P_0]\widetilde P_0.
\]

\begin{theorem}[Quantitative stability on regular regions]
Assume that:
\begin{enumerate}
\item \(L\) and \(\widetilde L\) are of class \(C^2\) on \(U\) and have the same constant zero-mode multiplicity \(m\);
\item there exists a common spectral-gap constant \(\gamma>0\) such that, for every \(x\in U\),
\[
\operatorname{Spec}(L(x)),
\operatorname{Spec}(\widetilde L(x))
\subset
\{0\}\cup[\gamma,\infty);
\]
\item \(L\) and \(\widetilde L\) have uniformly bounded \(C^2\) norms on \(U\).
\end{enumerate}
Then there exist constants \(C_0,C_1,C_2>0\) such that
\[
\|P_0-\widetilde P_0\|_{C^0(U)}
\le
C_0\|L-\widetilde L\|_{C^0(U)},
\]
\[
\|dP_0-d\widetilde P_0\|_{C^0(U)}
\le
C_1\|L-\widetilde L\|_{C^1(U)},
\]
and
\[
\|F-\widetilde F\|_{C^0(U)}
\le
C_2\|L-\widetilde L\|_{C^2(U)}.
\]
In particular, on a regular region with a uniform spectral gap, the zero-mode projection is Lipschitz stable with respect to \(C^0\)-perturbations of the operator family, and the curvature is Lipschitz stable with respect to \(C^2\)-perturbations of the operator family.
\end{theorem}

\begin{proof}[Proof sketch]
The zero-mode projection is expressed as the Riesz projection
\[
P_0
=
\frac{1}{2\pi i}
\int_\Gamma
(z-L)^{-1}\,dz.
\]
Hence the estimate for \(P_0-\widetilde P_0\) follows from the resolvent identity. Next, differentiating the Riesz projection formula gives the estimate for \(dP_0-d\widetilde P_0\). Finally, using the projection formula for curvature,
\[
F=P_0[dP_0,dP_0]P_0,
\]
the estimate for \(F-\widetilde F\) follows from the estimates for \(P_0\) and \(dP_0\).
\end{proof}

The constants in the estimates deteriorate as the spectral gap \(\gamma\) becomes smaller. Therefore, near the singular set, such uniform stability generally cannot be expected. This is not a defect of the theory, but reflects the loss of separation between zero modes and nonzero modes.

\subsection{Choice of the Ordinary Hodge Laplacian}

We finally explain why this paper uses the ordinary Hodge Laplacian rather than the persistent Laplacian.

If the goal is to realize fixed-birth persistent homology
\[
H_q^{b_0,d}(\lambda)
\]
as a zero-mode space, then the persistent Laplacian is the natural operator. The persistent Laplacian is designed to give an operator-theoretic realization of the persistent homology associated with the inclusion
\[
K(b_0,\lambda)\subset K(d,\lambda),
\]
and is therefore the most direct object for studying fixed-birth persistent feature spaces.

However, the curvature and holonomy introduced in this work are differential-geometric quantities associated with a family of fibers over parameter space. In particular, the curvature
\[
F=P_0[dP_0,dP_0]P_0
\]
is a local quantity measuring how the zero-mode projection \(P_0(d,\lambda)\) changes noncommutatively under infinitesimal variations in the \(d\)-direction and the \(\lambda\)-direction. From this viewpoint, the ordinary Hodge Laplacian, which is intrinsically attached to the complex \(K(d,\lambda)\) at each parameter point \((d,\lambda)\), is more natural for describing local transport geometry.

When the ordinary Hodge Laplacian is used, the fiber is
\[
E_0(d,\lambda)
=
\ker\Delta_q^{\mathrm{Hodge}}(d,\lambda)
\cong
H_q(K(d,\lambda)).
\]
This space is canonically determined by the complex at the point \((d,\lambda)\) itself. Therefore, the resulting zero-mode bundle is a family of instantaneous homology spaces placed over parameter space. In this case, the curvature directly reflects the local reorganization of the homology space of \(K(d,\lambda)\) itself.

By contrast, when the persistent Laplacian is used, the fiber is
\[
\ker\Delta_{q,b_0}^{\mathrm{pers}}(d,\lambda)
\cong
\operatorname{Im}
\bigl(
H_q(K(b_0,\lambda))
\to
H_q(K(d,\lambda))
\bigr).
\]
This space is not determined solely by the complex at \((d,\lambda)\), but by its relation to the fixed birth level \(b_0\). Thus, the zero-mode space of the persistent Laplacian is not a homology space locally attached to a parameter point, but rather a space involving the persistence window from the fixed birth level to the death level.

This distinction is important for the interpretation of curvature. The curvature obtained from the persistent Laplacian contains not only the local reorganization of the homology space, but also the change in feature selection imposed by the fixed birth level \(b_0\). Thus, when the curvature is large, it may reflect either a reorganization of the homological structure of the current complex \(K(d,\lambda)\), or a change in the selection of persistent features caused by the fixed birth constraint, or a mixture of both.

Moreover, the choice of \(b_0\) is not merely a gauge choice. A gauge transformation changes only the frame of the same vector bundle, whereas changing \(b_0\) generally changes the fiber
\[
\operatorname{Im}
\bigl(
H_q(K(b_0,\lambda))
\to
H_q(K(d,\lambda))
\bigr)
\]
itself. Therefore, the fixed birth level is not an auxiliary choice that only changes the representation of the connection; it is an additional structure that may change the bundle itself.

For this reason, if the goal is first to isolate the local zero-mode transport geometry over parameter space, it is theoretically more natural to use the ordinary Hodge Laplacian. The ordinary Hodge Laplacian is intrinsically determined by the complex at each parameter point and does not depend on the choice of a fixed birth level or a persistent upper operator. As a result, the curvature and holonomy obtained from it more directly reflect the transport, mixing, reorganization, and memory of the parameter-dependent homology spaces themselves.

Thus, the role of the two Laplacians can be summarized as follows. The persistent Laplacian is the natural operator for realizing fixed-birth persistent homology as a zero-mode space. On the other hand, the ordinary Hodge Laplacian gives an intrinsic zero-mode realization of the homology space at each parameter point, and is therefore a more standard and more interpretable object for detecting local and global transport geometry such as curvature and holonomy.

In this sense, the formulation in this paper uses the ordinary Hodge Laplacian rather than the persistent Laplacian. This does not diminish the importance of persistent homology. Rather, it clarifies the transport geometry of homological zero-mode spaces themselves before introducing the additional persistence window associated with a fixed birth level. Curvature describes local reorganization, while holonomy describes global memory accumulated along closed loops. This basic geometric mechanism appears most clearly in the zero-mode bundle associated with the ordinary Hodge Laplacian.

\section{Geometric Interpretation: Reorganization, Memory, and Dynamic Homological Geometry}

In this section, we explain the geometric meaning of the connection, curvature, and holonomy introduced in Section~3. For parameter-dependent topological data, the essential issue is not only which homological features exist at each parameter value, but also how the corresponding homological feature space moves over parameter space. The connection provides the rule of comparison, the curvature measures local reorganization, and the holonomy records the cumulative global effect along closed loops.

\subsection{Dynamic homological geometry}

In one-parameter persistent homology, barcodes and persistence diagrams provide a concise description of the birth and death of topological features. They indicate at which scales features appear and how long they persist, thereby encoding both existence and robustness. However, when an additional parameter is present, such as time or an external field, listing persistence diagrams at each parameter value is no longer sufficient to describe the full behavior of the underlying homological structure.

The reason is that, in the presence of an additional parameter, what evolves is not merely the birth--death data, but the homological feature space itself. Even when the Betti number remains constant, the zero-mode space \(E_0(d,\lambda)\) of the ordinary Hodge Laplacian may point in different directions inside the common ambient Hilbert space. Moreover, this direction may vary continuously with the parameters, mix locally, and accumulate into a nontrivial transformation after one circuit along a closed path. Therefore, to understand parameter-dependent homological structure, one must consider not only what exists at each point, but also how the feature spaces are transported from point to point.

From this perspective, the Hodge zero-mode bundle provides the dynamical object associated with parameter-dependent homology, and the connection specifies how to compare nearby fibers. The resulting curvature and holonomy describe local reorganization and global memory, respectively. In this way, the framework developed in this paper provides a language for understanding homology not as a mere sequence of static snapshots, but as a geometric process in which time evolution and scale dependence interact.

\subsection{Reorganization of homological feature spaces}

To understand this phenomenon intuitively, it is useful to consider the simplest nontrivial example. Suppose that, in some fixed homological degree, the homology is always two-dimensional, so that there are always ``two important features.'' Then the Betti number is constantly equal to \(2\), and each persistence diagram appears to indicate the existence of two dominant features at every parameter value.

Nevertheless, the internal structure of the homology space may still vary. At one parameter value, a natural basis may be
\[
(\psi_1,\psi_2),
\]
whereas at another parameter value, the more natural basis may be
\[
\left(\frac{\psi_1+\psi_2}{\sqrt2},\frac{\psi_1-\psi_2}{\sqrt2}\right).
\]
Moreover, after transport along a closed path, one may obtain an exchange such as
\[
(\psi_1,\psi_2)\mapsto (\psi_2,\psi_1).
\]

What changes here is not the number of features, but the internal structure of the zero-mode space, namely, which directions should be naturally associated with which features and how those directions mix. In this paper, we refer to such changes---recombination, mixing, and exchange of basis directions inside the homological feature space---as \emph{reorganization}.

Geometrically, such reorganization typically occurs where several features approach one another and can no longer be easily distinguished as independent objects. One may imagine, for example, two holes gradually approaching each other, or two loops becoming connected by a thin corridor and thus nearly merging. In such a situation, one may still be able to choose generators corresponding to the ``left hole'' and the ``right hole'' at one time and one scale, but after a slight change in time or scale, this distinction may no longer be natural.

In other words, what changes here is not the number of holes. Rather, what changes is which linear combinations should be regarded as the natural independent generators. At one parameter value, two generators may correspond to two spatially separated holes, while at another parameter value, the sum and difference of those generators may provide a more natural description. In this sense, reorganization concerns not the mere existence of features, but the rearrangement of the relations among them.

To make this more explicit, consider again generators
\[
\psi_1,\qquad \psi_2.
\]
If the two holes are well separated, then \(\psi_1\) and \(\psi_2\) can be naturally regarded as corresponding to two distinct features. But as the two holes approach one another, it may become more natural to use the combinations
\[
\frac{\psi_1+\psi_2}{\sqrt2},\qquad \frac{\psi_1-\psi_2}{\sqrt2}
\]
instead. This does not mean that the original features have disappeared. Rather, it means that the natural coordinate system for describing the features has changed. The connection introduced in this paper describes this change of natural coordinates as transport.

\subsection{Curvature and local sensitivity}

Curvature is a local indicator of reorganization because it measures the noncommutativity of infinitesimal transport in two parameter directions along the zero-mode subspace. As shown in Section~3, the curvature is given by
\[
F=dA+A\wedge A,
\]
or equivalently by
\[
F=P_0[dP_0,dP_0]P_0.
\]
The latter expression is particularly revealing. The differential \(dP_0\) describes how the zero-mode subspace changes under parameter variation, and the commutator \([dP_0,dP_0]\) measures the failure of two infinitesimal variations to commute. The final projections by \(P_0\) restrict the result back to the zero-mode subspace itself. Hence the curvature measures how the mismatch between infinitesimal changes appears as an internal mixing inside the homological feature space.

More intuitively, consider moving first slightly in the time direction and then slightly in the scale direction, and compare this with doing the same two moves in the opposite order. If the resulting zero-mode spaces differ even infinitesimally, then there is local twisting in the transport of the feature space. The curvature measures that infinitesimal discrepancy. If the curvature vanishes, then the two infinitesimal transports are locally compatible. If it is large, then the internal structure of the zero-mode space is being rapidly reorganized in that neighborhood.

Therefore, a curvature heatmap should be understood as a visualization of where, in time and scale, the internal structure of the homology space changes most strongly. It indicates not merely where the number of features changes, but where nontrivial changes occur inside the feature space itself.

Moreover, large curvature also suggests increased sensitivity of the zero-mode description, that is, reduced robustness. In the standard eigenstate formula for Berry curvature, the curvature is generally expressed as a sum of couplings to other modes weighted by the inverse square of the spectral gaps. Consequently, when the spectral separation between the zero modes and nearby nonzero modes becomes weak, the zero-mode projection \(P_0\) becomes more sensitive to parameter variation, and the curvature tends to become large.

It is important, however, that the proximity of nearby nonzero modes is not itself the reorganization. Reorganization itself is the internal mixing within the zero-mode subspace, while the proximity of nonzero modes is the mechanism that makes the geometry sharper and more sensitive. In this sense, large curvature has a unified interpretation: geometrically, it indicates strong internal reorganization of the zero-mode space; analytically, it indicates reduced robustness of the zero-mode projection. These are two aspects of the same geometry.

\subsection{Holonomy, global memory, and the relation to vineyards}

While curvature is a local quantity, holonomy is a global quantity obtained by accumulating local effects along a closed path. If one parallel-transports the zero-mode bundle along a closed curve \(\gamma\), one returns to the same fiber at the starting point. The resulting linear transformation
\[
U_\gamma
\]
is the holonomy.

If \(U_\gamma=I\), then the transport along that loop is globally trivial, and the feature space returns to its original state. If \(U_\gamma\neq I\), then even though the Betti number at the beginning and at the end is the same, the internal structure of the feature space has undergone a nontrivial transformation. This means that the homological information may look unchanged at the level of dimension, while internally it has accumulated mixing, exchange, or more general monodromy.

For this reason, holonomy may be regarded as an indicator of \emph{global memory}. In particular, for periodically evolving data, it provides a natural quantity measuring what remains after one full cycle. It therefore goes beyond the question of how many features are present at the end, and instead records how the feature space remembers the path it has traversed.

This viewpoint does not oppose vineyards. Vineyards provide an intuitive picture of how individual topological features evolve by following the trajectories of points in the persistence diagram. When features are sufficiently well separated, this viewpoint is very effective and gives a clear geometric interpretation of feature tracking.

By contrast, the framework of the present paper provides a complementary viewpoint for situations in which vineyards become harder to apply. When several features approach one another in the persistence diagram, when many short-lived features appear, or when pointwise matching becomes unstable, the essential issue is not merely that labels on individual features become ambiguous. Rather, the internal structure of the homological feature space itself is undergoing substantial change. The curvature and holonomy introduced here are designed precisely to capture that underlying geometry.

In this sense, vineyards ask which point corresponds to which point, whereas the present framework asks how the whole feature space moves. The former is well suited for local feature tracking, while the latter is sensitive to internal mixing and global memory. The two viewpoints should therefore be understood as complementary rather than competing.

To summarize, the main gain of the present framework is threefold. First, curvature detects where in time and scale the internal structure of the homology space changes strongly. Second, holonomy records what global effect remains after those local changes accumulate around a loop. Third, by showing that the difficulty of pointwise tracking arises from strong reorganization of the homological feature space, the framework provides a geometric explanation for why tracking becomes difficult.

Accordingly, while the Betti number tells us how many features there are, curvature and holonomy tell us how the feature space itself moves. For homology with an additional parameter, this distinction is essential.

\section{Numerical Experiments}

In this section, we demonstrate through five numerical experiments how the proposed framework visualizes the transport geometry of parameter-dependent homology and how it complements the standard description based on persistence diagrams. In our approach, at each scale and external parameter, the zero-mode space of the ordinary Hodge Laplacian is regarded as the homological feature space, and its dependence on the scale and the external parameter is interpreted as transport on the zero-mode bundle. In this setting, curvature represents local reorganization of the homological feature space, while holonomy represents global memory accumulated along closed loops. In particular, from the viewpoint of this paper, even when the Betti number remains constant, the internal structure of the zero-mode space may still rotate, mix, and reorganize continuously, so dimensional information alone is not sufficient to describe the essential change.

The experiments are organized as follows. In Experiment~1, we first confirm that when vineyard-style tracking is stably defined, vineyard monodromy is reproduced by zero-mode holonomy. This shows that the proposed method is not an unrelated alternative to vineyards, but rather a description on the zero-mode bundle that is consistent with vineyard-type global transport information. In Experiment~2, we then show that even in regions where persistence-diagram points approach one another and vineyard-style pointwise tracking becomes unstable, the curvature defined from the zero-mode projection remains stably computable. In Experiment~3, we show that two systems that are nearly indistinguishable by persistence diagrams and standard diagram-level quantities can nevertheless be distinguished by curvature. In Experiment~4, we show that even when pointwise tracking and local diagram comparison reveal little difference, one-cycle holonomy can detect a difference in global transport. Finally, in Experiment~5, we numerically verify the stability of curvature and holonomy under noise perturbations.

\subsection{Preparatory remark: pointwise tracking and PD drift}

Before presenting the experiments, we briefly explain the two diagram-based quantities used for comparison, namely pointwise tracking and PD drift. Both are defined only from persistence-diagram points, so they serve as natural baselines against which the proposed transport-geometric quantities can be compared.

At each sampled time \(t_j\), let
\[
p_1^{(j)},\; p_2^{(j)} \in \mathbb{R}^2
\]
denote the two persistence-diagram points with the largest lifetimes. Each point is written as
\[
p_a^{(j)}=(b_a^{(j)},d_a^{(j)}),
\]
which records the birth scale and death scale of one of the two most persistent features at time \(t_j\).

Pointwise tracking asks, in a very simple form, the following question: when we move from time \(t_j\) to \(t_{j+1}\), which of the two points at the new time should be regarded as the continuation of each previously tracked feature? To make this decision, we compare the two possible assignments between the two tracked points at the previous time,
\[
\widetilde p_1^{(j)},\; \widetilde p_2^{(j)},
\]
and the two candidate points at the next time,
\[
p_1^{(j+1)},\; p_2^{(j+1)}.
\]
Specifically, using the Euclidean distance in the persistence diagram, we define the matching cost
\[
C_{ab}^{(j)}=\|\widetilde p_a^{(j)}-p_b^{(j+1)}\|_2^2.
\]
We then choose the assignment with the smaller total cost. In other words, tracking is based on the idea that a feature should continue to the nearest available diagram point at the next time. When the two candidate points are well separated, this rule is usually almost unambiguous. However, when the candidate points come very close to each other, the identity assignment and the swapped assignment can have nearly the same cost, and the labeling becomes unstable.

By contrast, PD drift is a different quantity. It does not compare consecutive times within the same system, but instead compares two persistence diagrams at the same time, typically a baseline diagram and a noisy one. Let
\[
p_1^{(j)},p_2^{(j)}
\]
be the top two points of the baseline diagram at time \(t_j\), and let
\[
q_1^{(j)},q_2^{(j)}
\]
be the corresponding top two points of the perturbed diagram. Then the PD drift is defined as the minimum average matching distance
\[
\mathrm{PDdrift}(t_j)
=
\frac{1}{2}
\min_{\sigma\in S_2}
\sum_{a=1}^2
\|p_a^{(j)}-q_{\sigma(a)}\|_2.
\]
Thus, PD drift measures how much the two most persistent diagram points move under perturbation. A small value means that the noisy diagram remains close to the baseline one, whereas a larger value indicates a larger diagram-level deformation.

In summary, pointwise tracking compares persistence-diagram points across neighboring times in order to maintain feature labels, whereas PD drift compares persistence-diagram points between two datasets at the same time in order to measure diagram displacement. These two quantities provide intuitive diagram-based baselines, while the proposed curvature and holonomy are designed to capture changes in the transport structure of the homological feature space beyond such pointwise comparisons.

\subsection{Experiment 1: Vineyard-like Monodromy and Zero-Mode Holonomy}

In Experiment~1, we examine the relation between vineyard-type monodromy and the holonomy of the zero-mode bundle. This experiment is motivated by the viewpoint of Braiding Vineyards, where persistence diagrams along a periodic one-parameter family are stacked to form a vineyard, and identifying the two ends of the period gives a closed vineyard. In such a closed vineyard, following the family for one period may permute the persistence-diagram points nontrivially; this is interpreted as vineyard monodromy \cite{chambers2026braiding}. For the detailed geometric construction of braided vineyards, we refer the reader to \cite{chambers2026braiding}. Here, we use a simplified vineyard-like dataset whose tracked birth and death coordinates exhibit the same type of cyclic behavior.

Figure~\ref{fig:exp5-vineyard-tracking} shows the vineyard-like tracking used in this experiment. The upper panel shows the birth coordinates of the three tracked features, and the lower panel shows the corresponding death coordinates. The red dotted curve represents the elder vine, while the shaded regions indicate crossing windows. The three dominant non-elder vines vary smoothly over one period and approach one another near the crossing windows. The data are constructed so that, after one full period, the three dominant vines undergo a cyclic relabeling.

\begin{figure}[t]
    \centering
    \includegraphics[width=\linewidth]{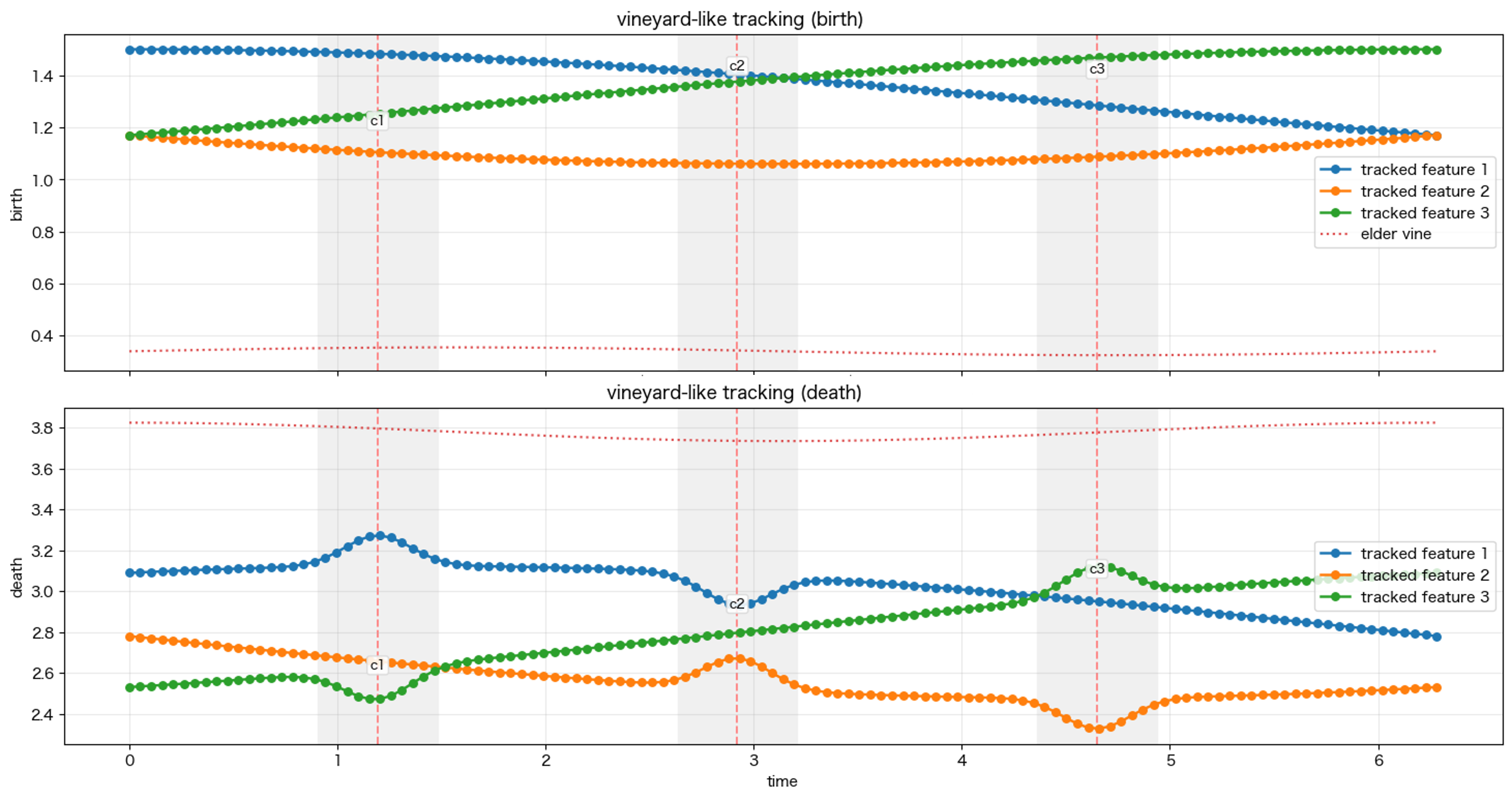}
    \caption{
    Vineyard-like tracking used in Experiment~1.
    The resulting vineyard-like tracking, including the elder vine and crossing windows.
    }
    \label{fig:exp5-vineyard-tracking}
\end{figure}

We then construct the ordinary Hodge Laplacian at each parameter value and transport its zero-mode space along the time direction over one full cycle. In order to compare the result with the vineyard-side monodromy, we focus on the non-elder part selected from the Hodge zero-mode space using the dominant tracked features. The one-cycle holonomy obtained from the zero-mode transport induces the nearest permutation
\[
    (1,2,0).
\]
This agrees with the cyclic endpoint permutation expected from the vineyard-like tracking. Thus, the global relabeling visible at the level of persistence-diagram points is reproduced as the holonomy of the Hodge zero-mode space.

Figure~\ref{fig:exp5-permutation-holonomy} compares the two descriptions directly. On the vineyard side, following the three non-elder vines over one period sends
\[
    v_1 \mapsto v_2,\qquad
    v_2 \mapsto v_3,\qquad
    v_3 \mapsto v_1.
\]
On the holonomy side, the corresponding zero-mode transport gives the same permutation matrix. This shows that the endpoint permutation of the braided vineyard and the one-cycle holonomy of the Hodge zero-mode bundle encode the same global transport information in this example.

\begin{figure}[t]
    \centering
    \includegraphics[width=\linewidth]{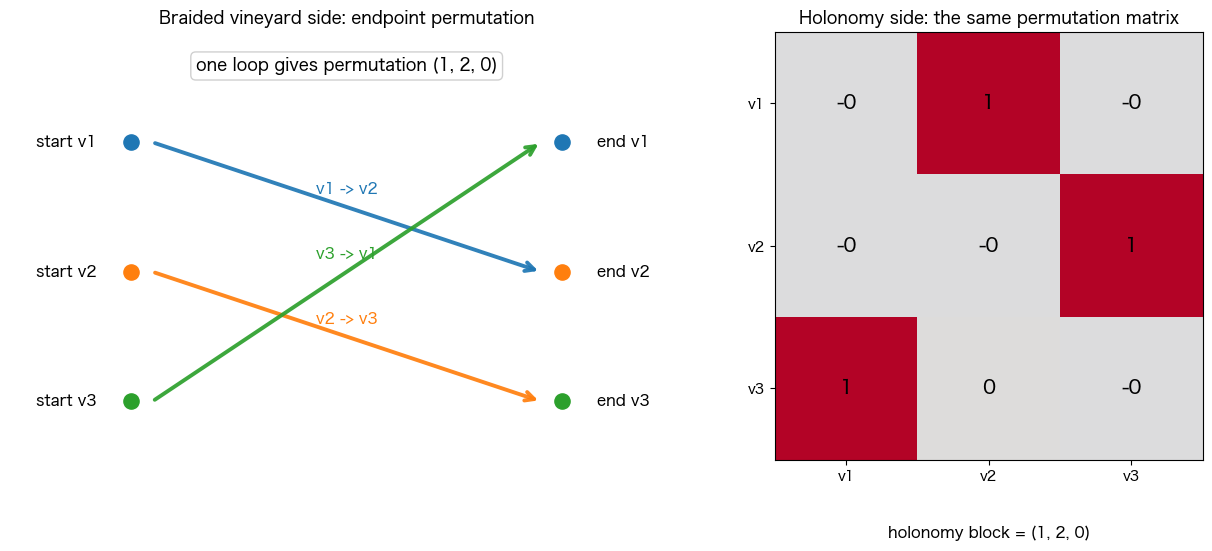}
    \caption{
    Endpoint permutation and holonomy in Experiment~1.
    One loop sends \((v_1,v_2,v_3)\) to \((v_2,v_3,v_1)\), and the zero-mode holonomy gives the same permutation matrix.
    }
    \label{fig:exp5-permutation-holonomy}
\end{figure}

Moreover, after removing the trivial direction corresponding to the sum of the three non-elder components, the induced action on the two-dimensional quotient space is a rotation by
\[
    -120^\circ .
\]
This is the standard two-dimensional geometric representation of an order-three cyclic monodromy. Equivalently, applying the holonomy three times returns the quotient zero-mode space to its original state.

This experiment clarifies the relation between vineyards and the gauge-geometric framework proposed in this paper. Vineyard-style tracking describes how individual persistence-diagram points move and how their labels are permuted after one period. By contrast, the present framework describes the transport of the Hodge zero-mode space itself. In the vineyard-like situation considered here, the cyclic permutation of diagram points is lifted to a nontrivial linear holonomy on the zero-mode space. Therefore, the proposed holonomy should not be regarded as a replacement for vineyard monodromy, but rather as its zero-mode-bundle counterpart.

\subsection{Experiment 2: Where Pointwise Tracking Becomes Unstable}

In Experiment~2, we use a time evolution of double circles in which two loops approach each other and then separate again\ref{exp1_data}, and track the top two persistence features in the persistence diagram by vineyard-style successive matching. In pointwise tracking, one assigns points between consecutive times by minimizing a matching cost. However, when two important features approach each other, the difference between the identity assignment cost and the swap assignment cost becomes extremely small, and it becomes impossible to determine stably which point should be regarded as feature~1 and which as feature~2. In such regions, pointwise tracking becomes unstable both numerically and conceptually (see Figs.~\ref{exp2_vine}, \ref{exp2_amb}).

\begin{figure}[tbp]
\begin{center}
\includegraphics[width=140mm]{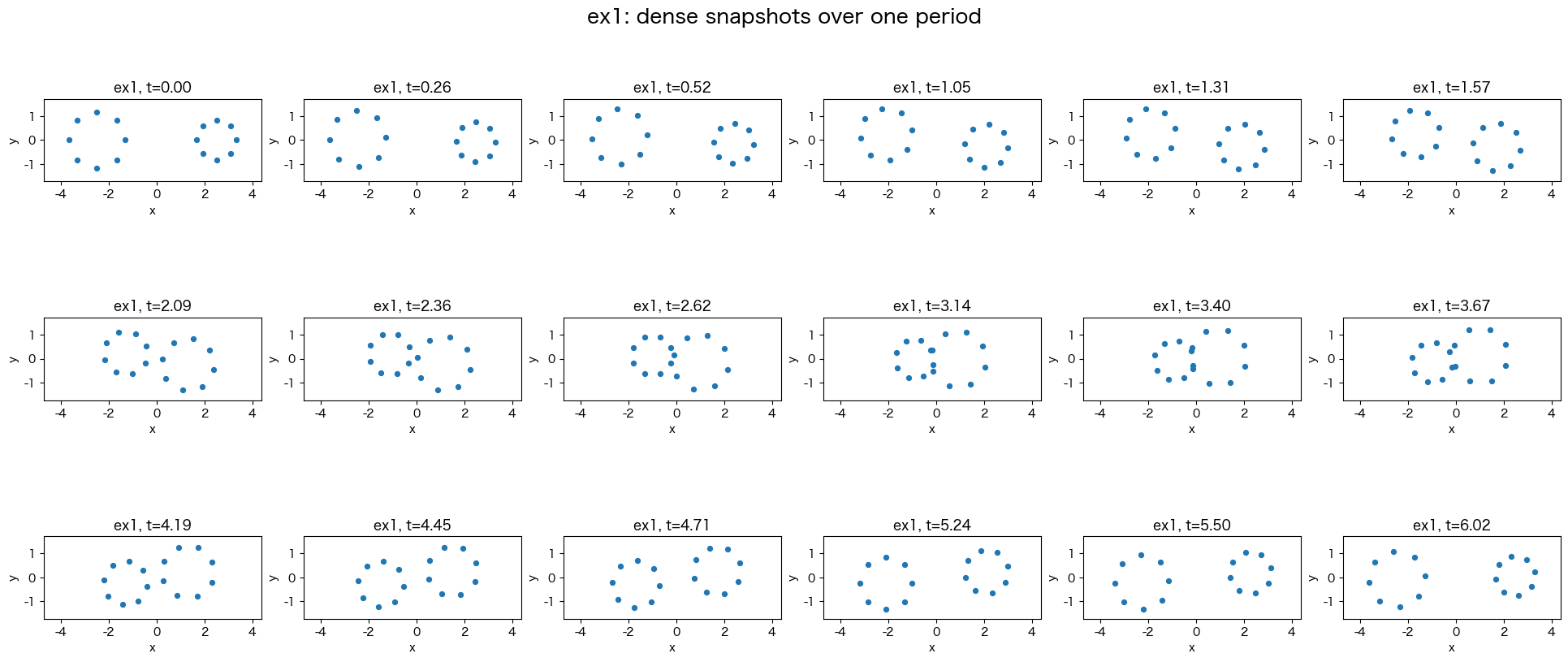}
\caption{Dense snapshots over one period for Experiment~3, approaching double circles.}
\label{exp1_data}
\end{center}
\end{figure}

\begin{figure}[tbp]
\begin{center}
\includegraphics[width=110mm]{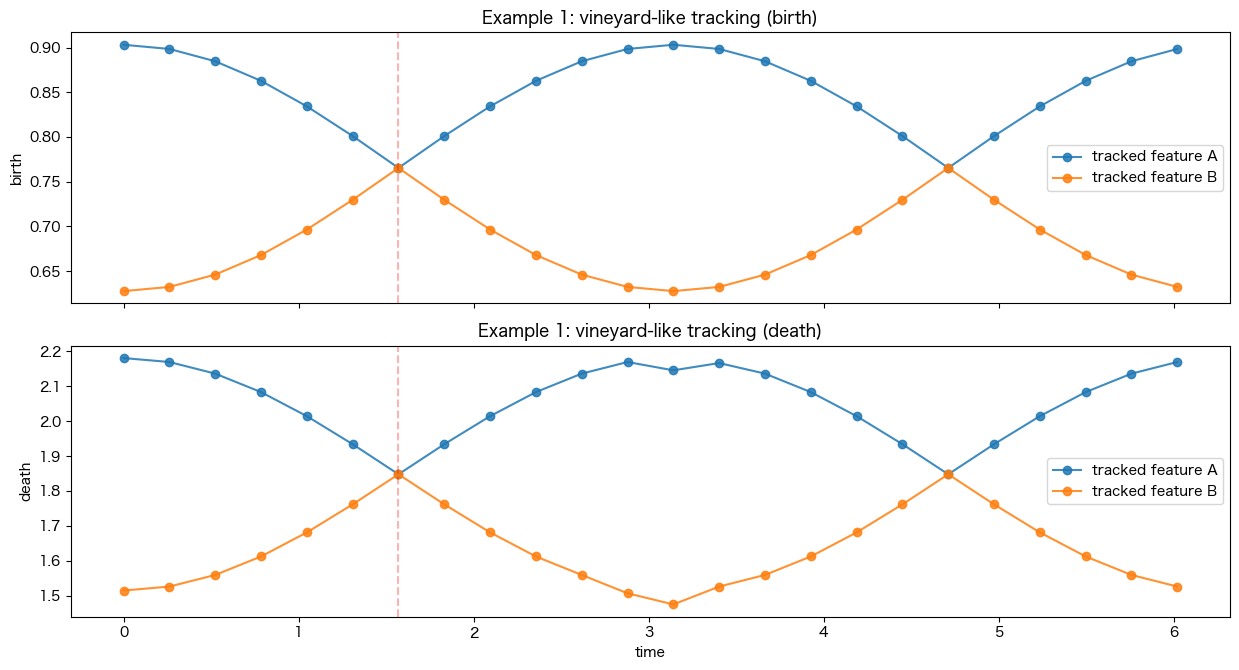}
\caption{Vineyard-style tracking of the top two persistence points in Experiment~2.}
\label{exp2_vine}
\end{center}
\end{figure}

\begin{figure}[tbp]
\begin{center}
\includegraphics[width=100mm]{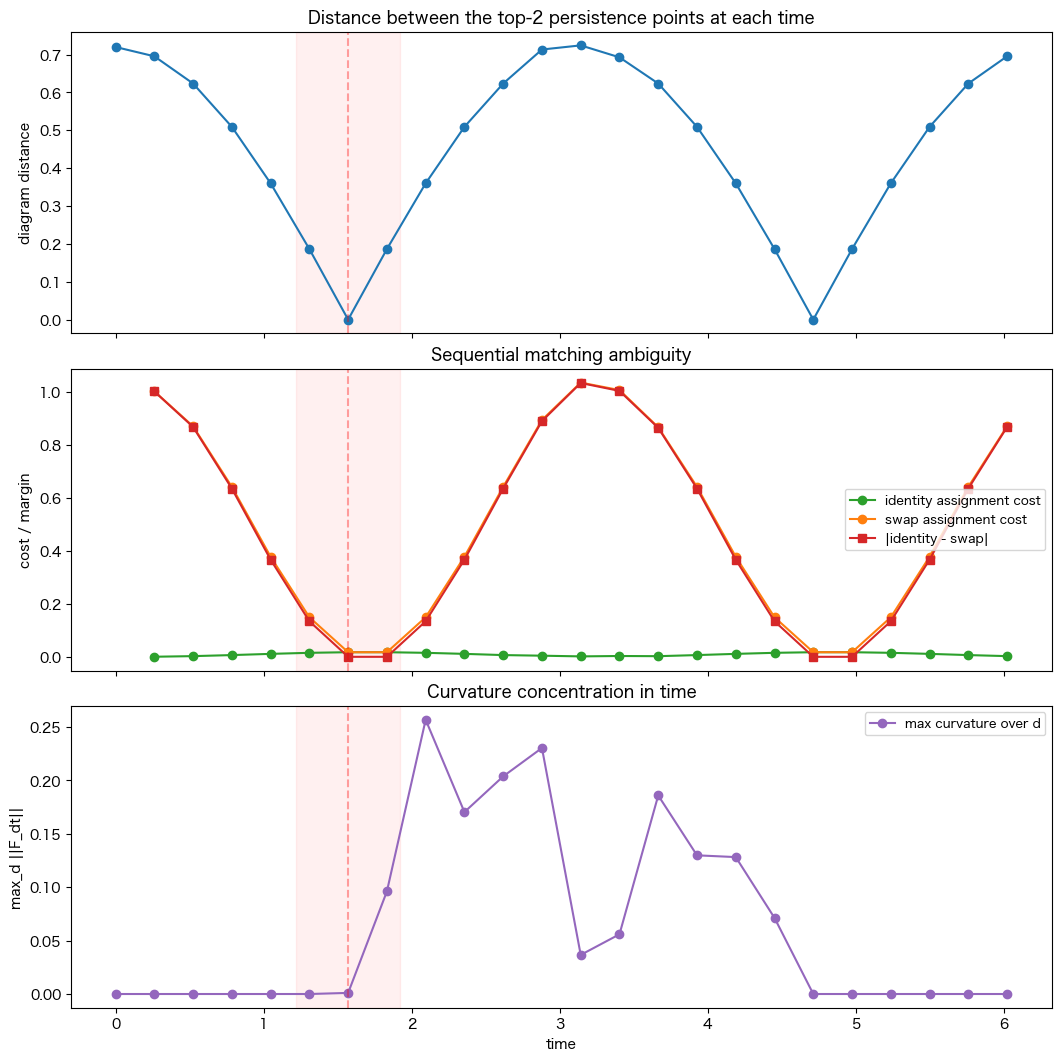}
\caption{Time series of three quantities in Experiment~2: the distance between the top two persistence points, the tracking ambiguity margin, and the maximum curvature. The distance and ambiguity margin become nearly zero at the swap time, while the curvature attains its peak in the same region.}
\label{exp2_amb}
\end{center}
\end{figure}

By contrast, the curvature used in the present work does not depend on any discrete correspondence between persistence-diagram points. Instead, it is computed directly from the zero-mode projection embedded in a common ambient space. Therefore, even in regions where vineyard-style pointwise tracking essentially loses its meaning, the curvature can still be evaluated. Indeed, even near the time where the top two persistence-diagram points nearly merge and pointwise tracking becomes ambiguous, the curvature exhibits a clear localized peak and continuously describes the strength of reorganization of the homological feature space. The crucial point is that curvature is not a quantity that judges whether a discrete labeling succeeds or fails; rather, it measures changes in the internal structure of the feature space itself.

More importantly, the curvature peak coincides with the region where the pointwise tracking swap occurs, or equivalently where the assignment ambiguity becomes extremely small (Fig.~\ref{exp2_amb}). This indicates that the instability of pointwise tracking is not merely a superficial consequence of the proximity of points in the persistence diagram, but is rooted in a strong reorganization of the homological feature space itself. In other words, the reason why it becomes unclear which point should correspond to which is that the internal zero-mode space is actually undergoing strong mixing, and the natural basis directions are changing rapidly.

In this sense, curvature should be understood not merely as an auxiliary quantity, but as a geometric explanation of why vineyard-style tracking becomes unreliable. As discussed in Section~3, curvature measures the noncommutativity of infinitesimal transport, that is, the strength of local reorganization. Experiment~2 shows that this theoretical interpretation has concrete explanatory power in an actual numerical example. Thus, this experiment provides a deeper geometric explanation of tracking instability beyond the usual description in terms of proximity of diagram points alone.

\subsection{Experiment 3: Similar Persistence Diagrams but Different Curvature}

In Experiment~3, we considered the time evolution of double circles in which two loops approach each other and then separate again (Fig.~\ref{exp1_data}). As a comparison, we prepared a size-only control in which only the overall scale changes over time (Fig.~\ref{exp1_data_}). The persistence diagrams of these two systems are very similar at each time, so at the diagram level they are almost indistinguishable (Figs.~\ref{exp1_paesis}, \ref{exp1_bottl}). However, from the viewpoint of the present work, what matters is not merely the similarity of the birth--death information at each time, but rather how the homological feature space is transported along the scale direction and the time direction.

\begin{figure}[tbp]
\begin{center}
\includegraphics[width=140mm]{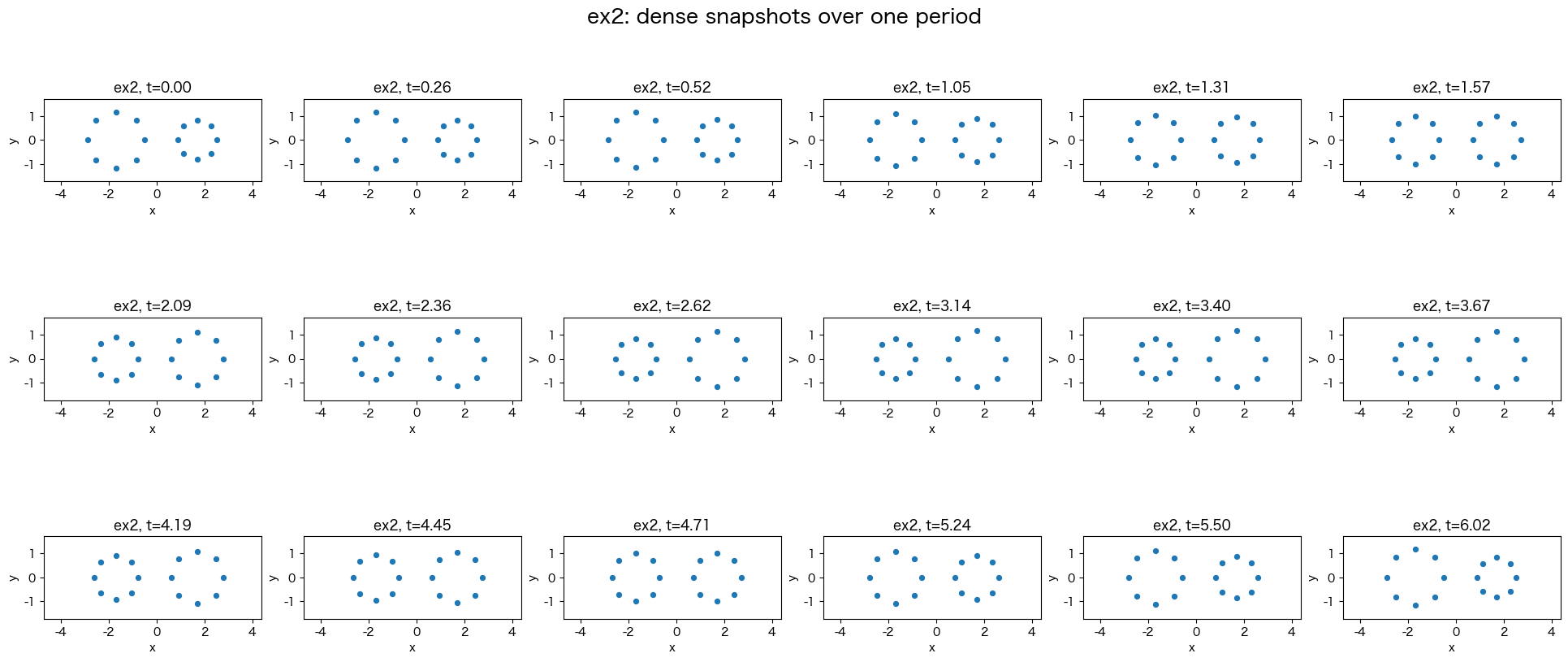}
\caption{Dense snapshots over one period for the size-only control in Experiment~3.}
\label{exp1_data_}
\end{center}
\end{figure}

As shown in Fig.~\ref{exp1_heat}, the curvature in the approaching double-circle system is clearly concentrated in a localized region, whereas in the size-only control it is almost completely zero. Thus, although the two systems look very similar in terms of persistence diagrams, they are fundamentally different from the viewpoint of transport geometry of the zero-mode space. In the control system, the near-vanishing curvature indicates that the transport is almost commutative and that there is essentially no internal reorganization of the homological feature space. By contrast, in the approaching double-circle system, even though the Betti number itself changes little, the internal structure of the zero-mode space undergoes substantial mixing and reorganization, which is visualized by the curvature map.

\begin{figure}[tbp]
\begin{center}
\includegraphics[width=140mm]{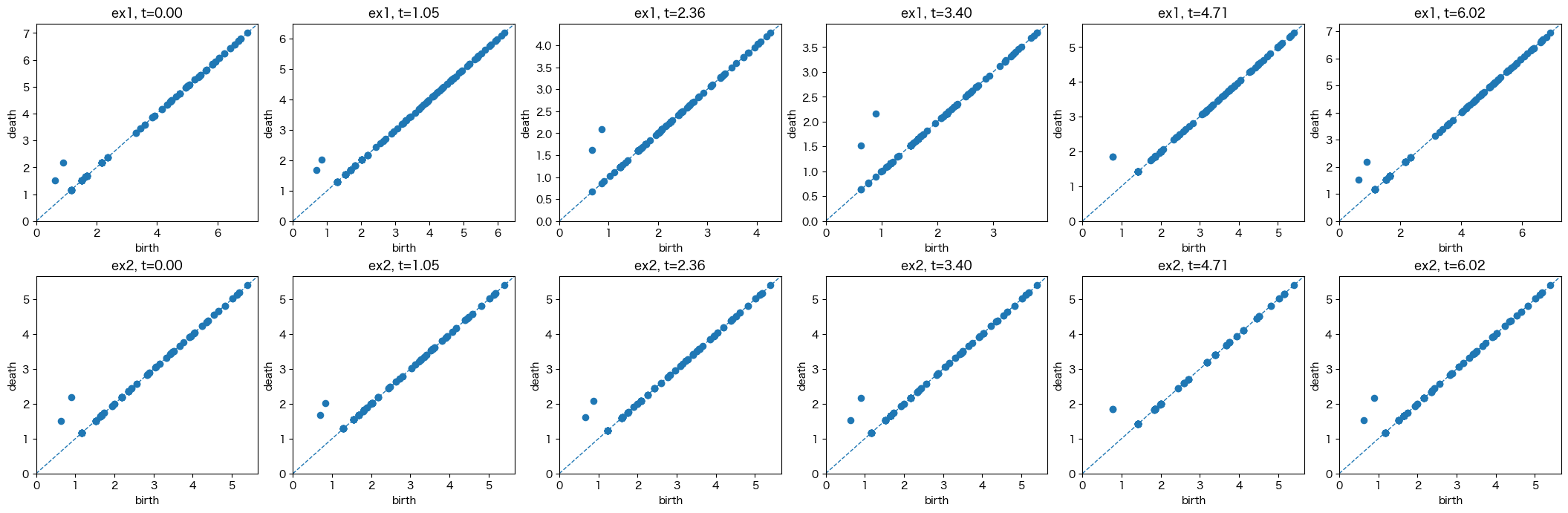}
\caption{Time evolution of the persistence diagrams for the two systems in Experiment~3.}
\label{exp1_paesis}
\end{center}
\end{figure}

\begin{figure}[tbp]
\begin{center}
\includegraphics[width=110mm]{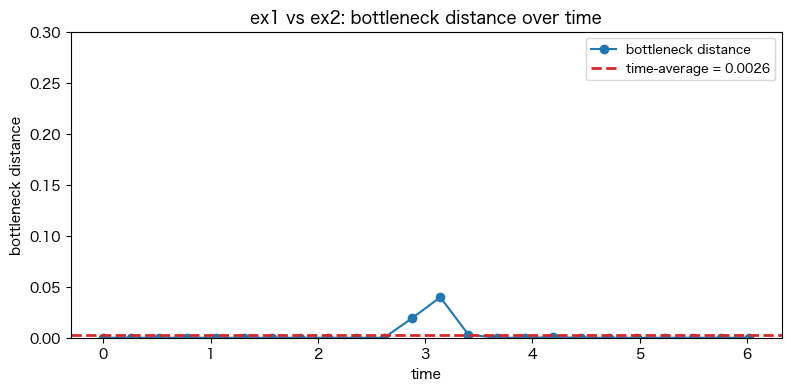}
\caption{Time evolution of the bottleneck distance between the two systems in Experiment~3.}
\label{exp1_bottl}
\end{center}
\end{figure}

\begin{figure}[tbp]
\begin{center}
\includegraphics[width=130mm]{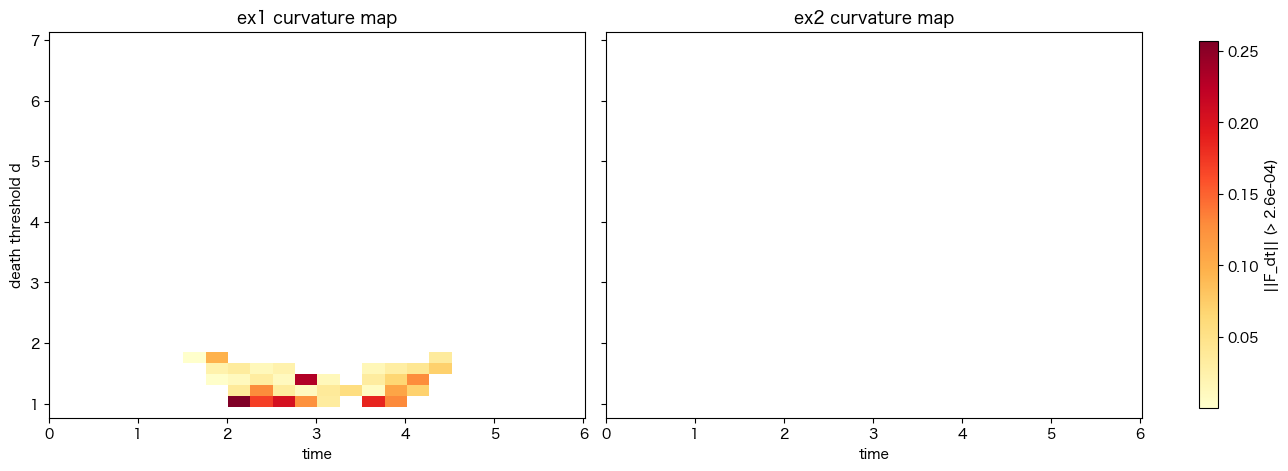}
\caption{Curvature heatmaps for the two systems in Experiment~3.}
\label{exp1_heat}
\end{center}
\end{figure}

This result demonstrates that closeness of persistence diagrams does not imply closeness of transport geometry. In other words, there are structural differences that are invisible at the level of birth--death data alone, and curvature plays an essential role in detecting them. The curvature introduced in Section~3,
\[
F = P[dP,dP]P,
\]
was defined as a measure of the noncommutativity of changes of the zero-mode projection, and Experiment~3 shows numerically that this quantity visualizes local reorganization of homological feature spaces. Therefore, this experiment provides an explicit example showing that the proposed framework does not replace persistence diagrams, but rather complements them by capturing transport-structural differences that are difficult to see from diagram data alone.

\subsection{Experiment 4: Holonomy Reveals Global Differences Invisible to Pointwise Tracking}

In Experiment~4, we used periodically deforming dumbbell-type point-cloud data (Fig.~\ref{exp3_data3}). More specifically, we considered a time series containing two outer loop structures together with an additional loop-like structure in the middle. As a comparison, we considered another dataset in which, as shown in Fig.~\ref{exp3_data4}, the middle loop appears to rotate over time. The point clouds evolve smoothly and continuously in time, and from the viewpoint of the time series of persistence diagrams the two systems are difficult to distinguish locally. The purpose of this experiment is to examine whether holonomy, as the accumulated transport over one full cycle, can detect a difference between two systems that look very similar at the local diagram level.

\begin{figure}[tbp]
\begin{center}
\includegraphics[width=140mm]{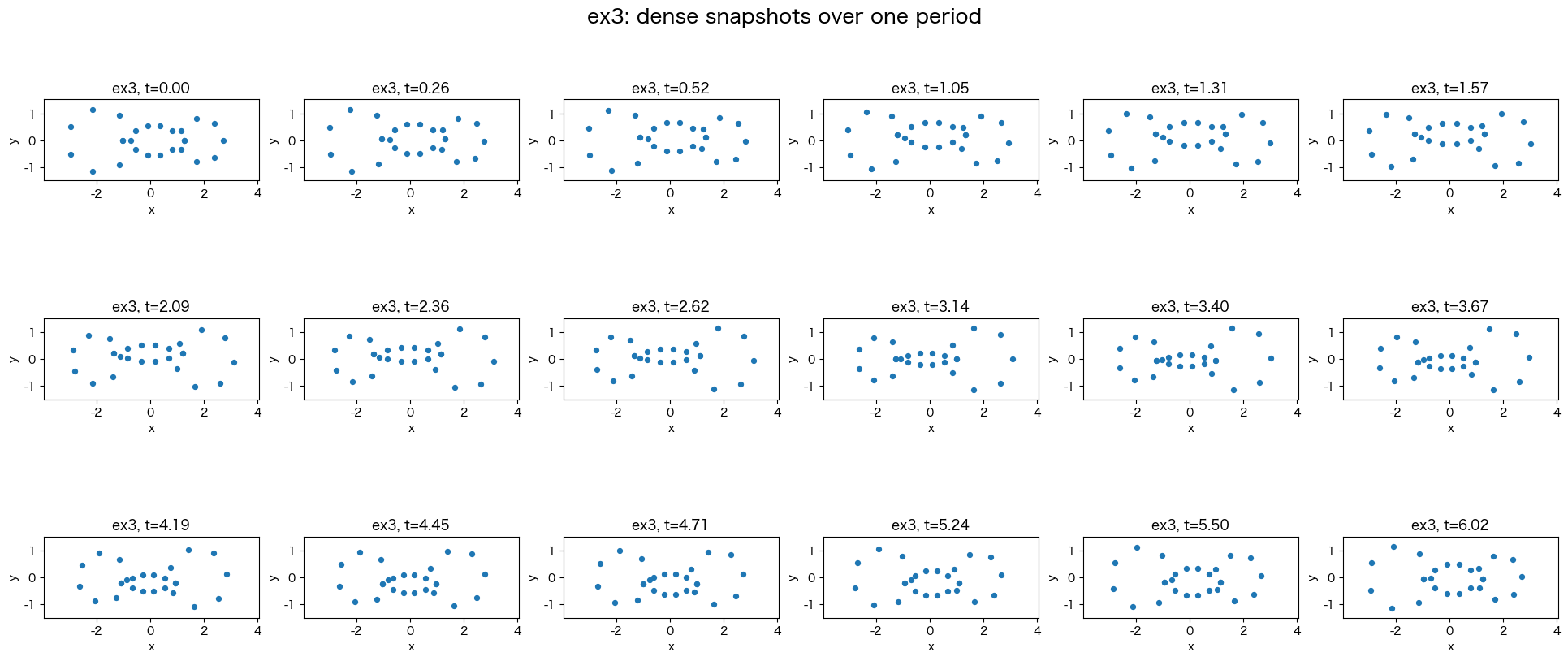}
\caption{Dense snapshots over one period for the first system in Experiment~4.}
\label{exp3_data3}
\end{center}
\end{figure}

\begin{figure}[tbp]
\begin{center}
\includegraphics[width=140mm]{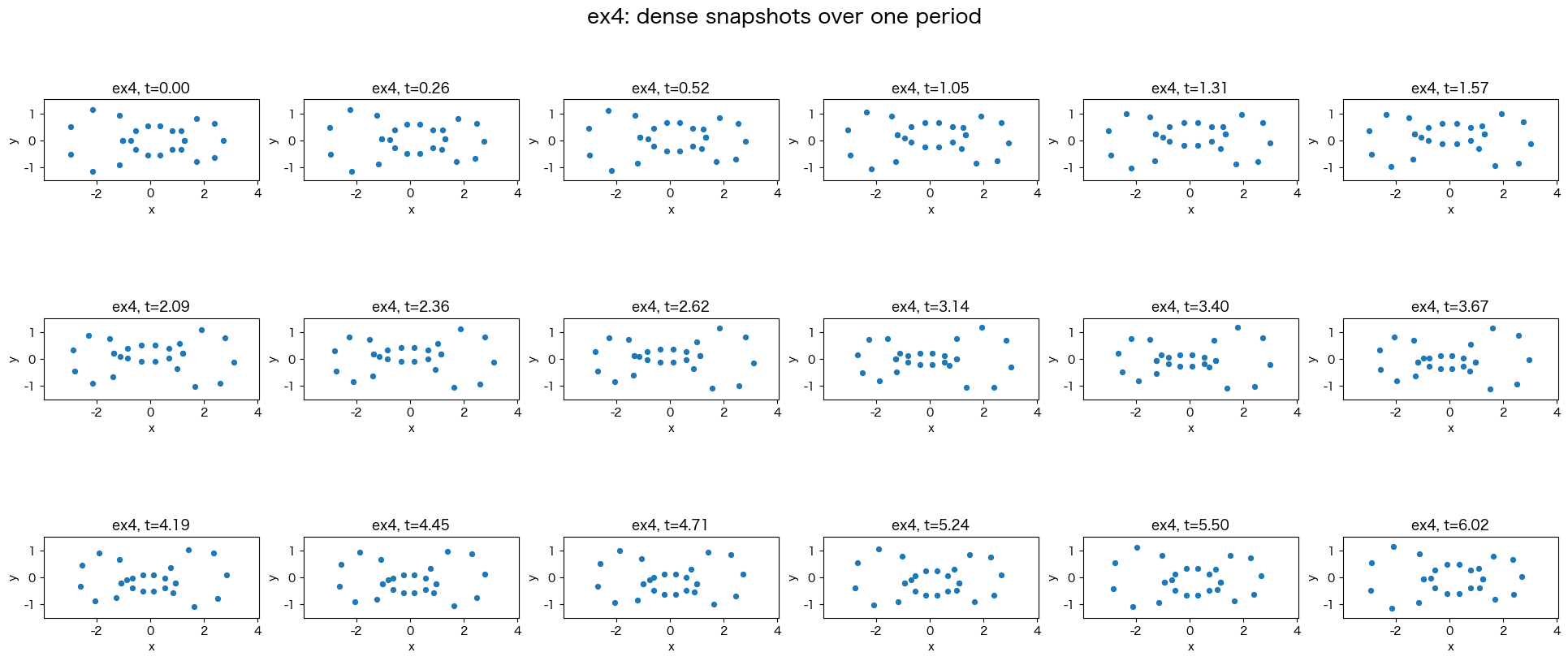}
\caption{Dense snapshots over one period for the second system in Experiment~4.}
\label{exp3_data4}
\end{center}
\end{figure}

We first applied the standard vineyard-style successive matching, that is, pointwise tracking based on local distances between persistence-diagram points. From this viewpoint, both systems have swap count zero, so no explicit local label exchange is observed (Fig.~\ref{exp3_vine}). Moreover, the mean diagram distance between the two systems is relatively small (Fig.~\ref{exp3_bottl}). Thus, based only on persistence-diagram-centered local tracking, the two systems do not appear to be substantially different.

\begin{figure}[tbp]
\begin{center}
\includegraphics[width=130mm]{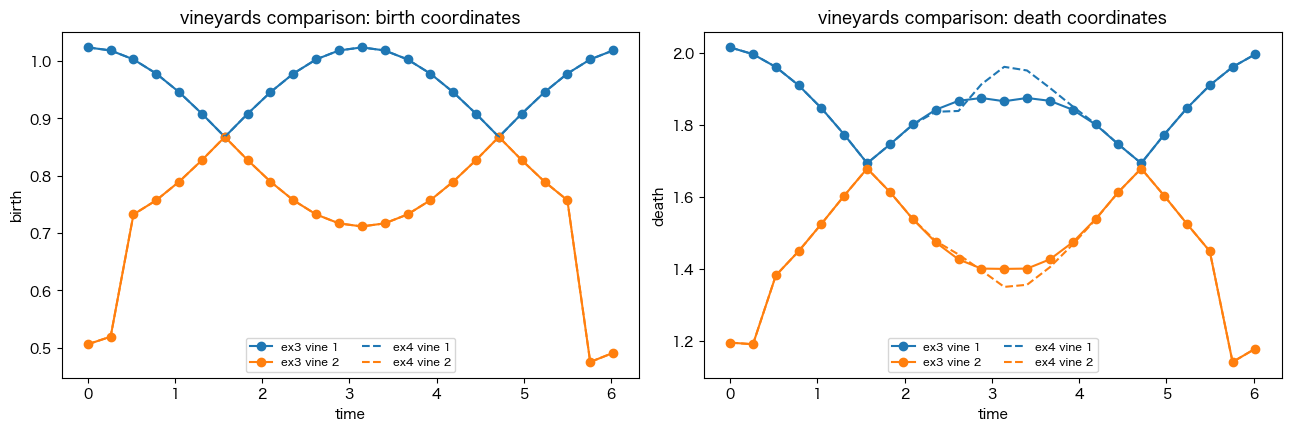}
\caption{Vineyard-style tracking of the top two persistence points in the two systems of Experiment~4.}
\label{exp3_vine}
\end{center}
\end{figure}

\begin{figure}[tbp]
\begin{center}
\includegraphics[width=130mm]{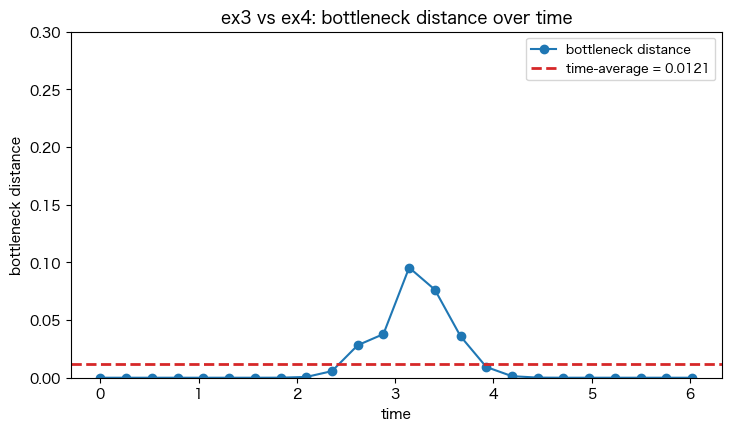}
\caption{Time evolution of the bottleneck distance between the two systems in Experiment~4.}
\label{exp3_bottl}
\end{center}
\end{figure}

By contrast, when we compare the one-cycle holonomy
\[
U_{\mathrm{cycle}}
=
Q_{N_t-1}Q_{N_t-2}\cdots Q_1Q_0,
\]
constructed by accumulating the discrete parallel transports of the Hodge zero-mode space over one full cycle, a clear difference emerges. For the first system, the one-cycle holonomy is
\[
U_{\mathrm{cycle}}(\mathrm{ex3})
=
\begin{pmatrix}
-0.6207 & -0.7840\\
-0.7840 & \phantom{-}0.6207
\end{pmatrix},
\]
whereas for the second system it is
\[
U_{\mathrm{cycle}}(\mathrm{ex4})
=
\begin{pmatrix}
\phantom{-}0.9795 & -0.2013\\
\phantom{-}0.2013 & \phantom{-}0.9795
\end{pmatrix}.
\]
These matrices are clearly different in form, showing that even when the point-cloud sequence and the persistence-diagram sequence look similar, the transformation undergone by the homological feature space after one full cycle can be fundamentally different.

To quantify the nontriviality of the holonomy, we also compare the Frobenius norm
\[
\|U_{\mathrm{cycle}}-I\|_F.
\]
We obtain
\[
\|U_{\mathrm{cycle}}(\mathrm{ex3})-I\|_F = 2.0000,
\qquad
\|U_{\mathrm{cycle}}(\mathrm{ex4})-I\|_F = 0.2862.
\]
Thus, the first system accumulates a strongly nontrivial transport over one cycle, whereas the second system remains much closer to the identity. Furthermore, the difference between the two holonomies is itself large:
\[
\|U_{\mathrm{cycle}}(\mathrm{ex3})-U_{\mathrm{cycle}}(\mathrm{ex4})\|_F = 2.0000.
\]
Therefore, at the level of one-cycle holonomy, the two systems are numerically well separated (see Figs.~\ref{exp3_hol}, \ref{exp3_vec}).

\begin{figure}[tbp]
\begin{center}
\includegraphics[width=120mm]{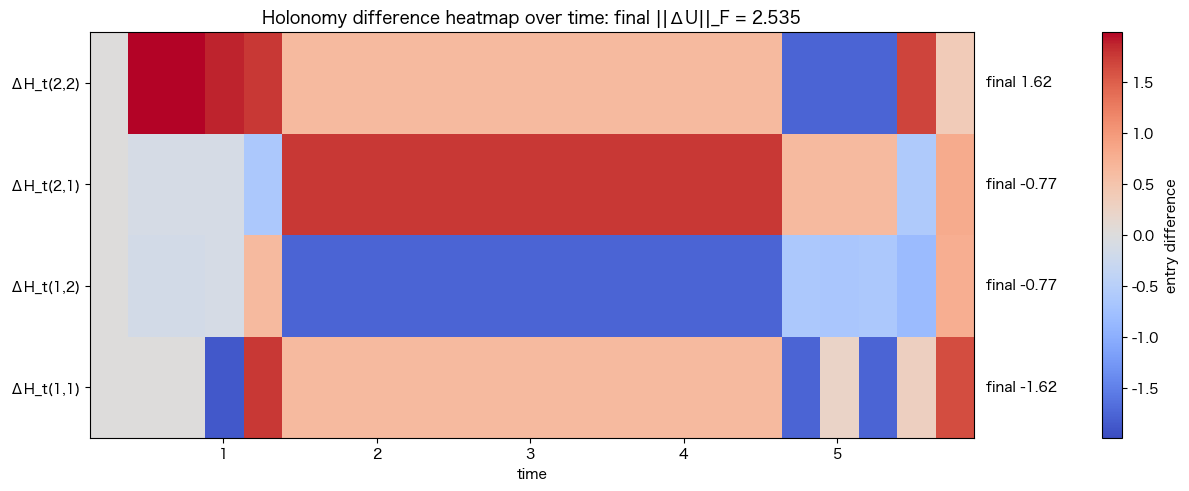}
\caption{One-cycle holonomy and tracking comparison for the two systems in Experiment~4.}
\label{exp3_hol}
\end{center}
\end{figure}

\begin{figure}[tbp]
\begin{center}
\includegraphics[width=\textwidth]{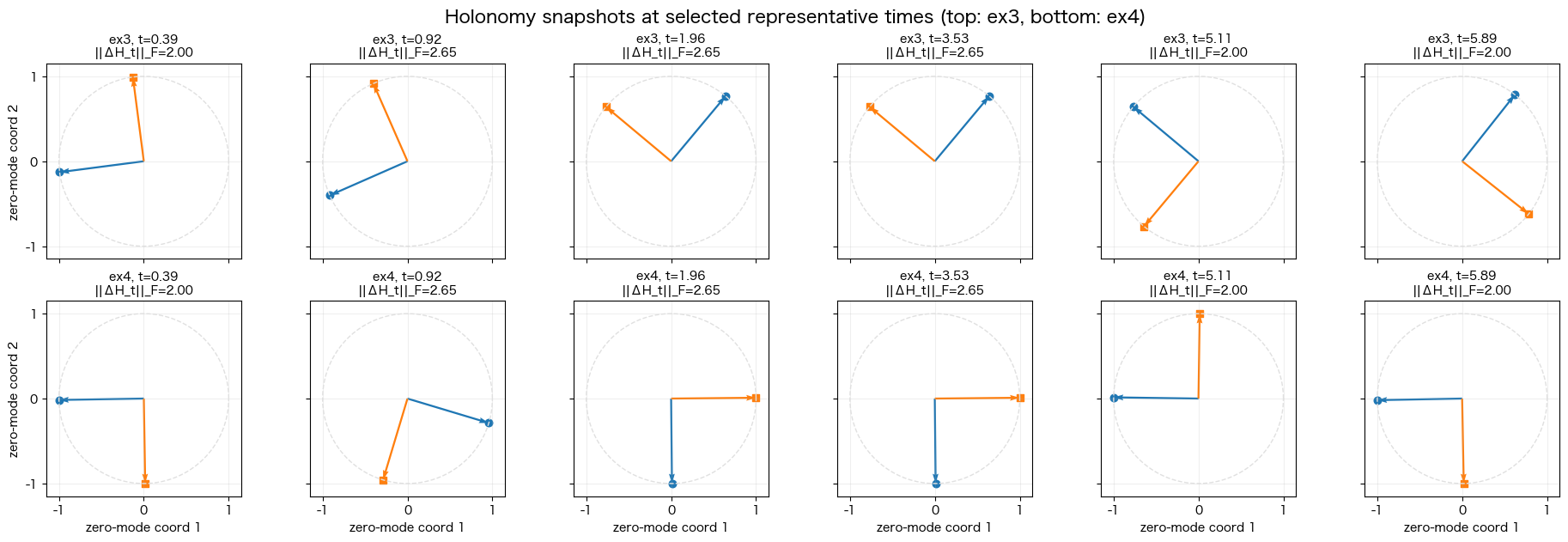}
\caption{Holonomy acting on zero-mode space for the two systems in Experiment~4.}
\label{exp3_vec}
\end{center}
\end{figure}

Accordingly, the essential message of Experiment~4 is that the local question asked by pointwise tracking---``which point continues to which point?''---is fundamentally different from the global question asked by holonomy---``what linear transformation does the homological feature space undergo after one full cycle?'' Even when the former shows almost no difference, the latter can reveal a substantial difference in global memory. This provides a direct numerical illustration of the theoretical meaning of holonomy introduced in Section~3, namely, as global memory accumulated along closed loops.

In Experiment~4, the holonomy is computed not from an arbitrary labeling of persistence-diagram points, but from the transport of a persistence-selected subspace inside the Hodge zero-mode space. More precisely, the persistence diagram is used only to select the dominant long-lived features at the chosen scale, and their representative cycles are projected onto the Hodge zero-mode space. After orthonormalization, these vectors define a rank-two selected frame, and the one-cycle holonomy is computed from the overlap matrices between these selected zero-mode frames. Thus, the holonomy in Experiment~4 should be interpreted as the Berry holonomy of a persistence-selected rank-two subbundle, rather than necessarily that of the full zero-mode bundle.

\subsection{Experiment 5: Stability of Curvature Through the Hodge Laplacian}

Finally, in Experiment~5, we examined how stable the curvature is under noise. In the theoretical part of this paper, we showed that on regular regions, as long as a sufficient spectral gap is maintained, the zero-mode projection and the curvature derived from it are locally stable under perturbations of the ordinary Hodge Laplacian. In Experiment~5, we tested how this theoretical expectation appears in actual discrete computations using multiple noisy realizations. Here we focus on the stability of curvature.

First, quantities based directly on persistence diagrams exhibit a very clean proportional relationship with the noise strength: as the noise increases, the diagram-level discrepancy also increases almost proportionally (Fig.~\ref{exp4}). This gives a natural baseline. In the present framework, however, we want to show that geometric quantities such as curvature are also stable, not directly at the diagram level, but through the more fundamental operator-level quantity given by the Hodge Laplacian.

\begin{figure}[tbp]
\begin{center}
\includegraphics[width=\textwidth]{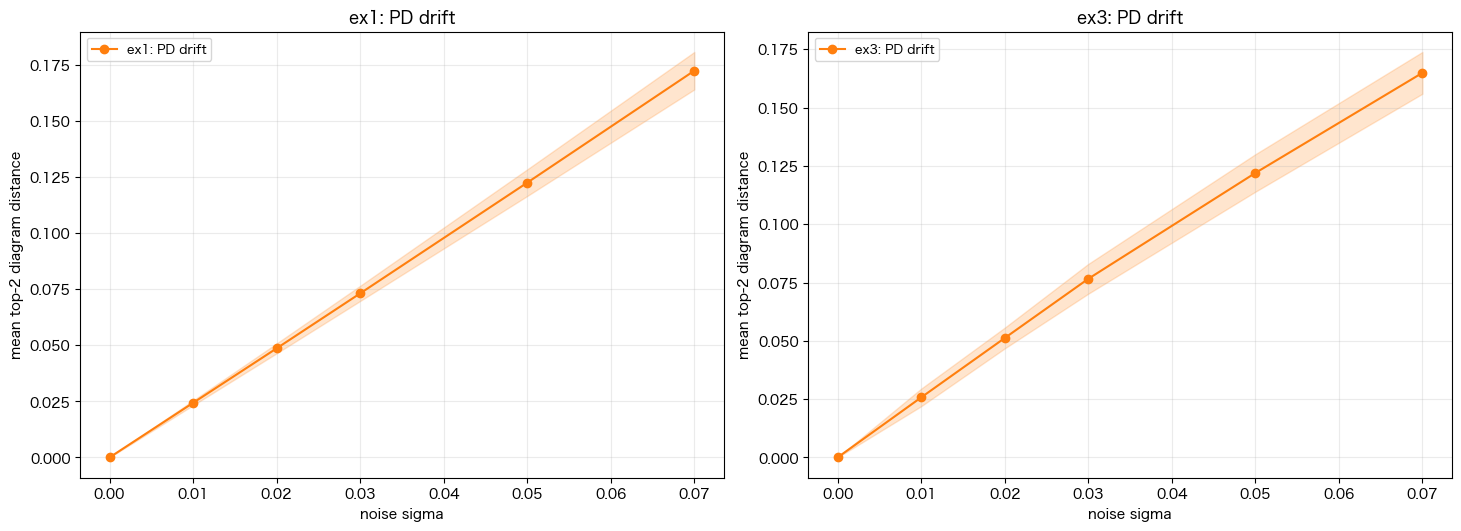}
\caption{Noise dependence of PD drift and swap count.}
\label{exp4}
\end{center}
\end{figure}

\begin{figure}[tbp]
\begin{center}
\includegraphics[width=\textwidth]{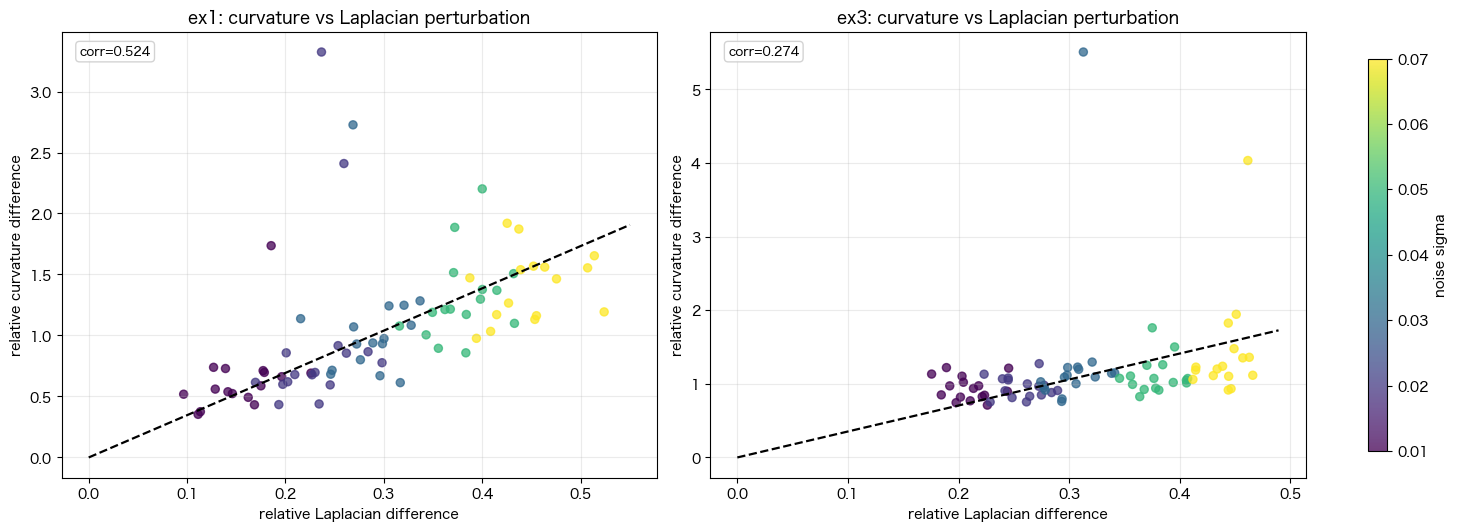}
\caption{Noise dependence of curvature versus Hodge Laplacian perturbation.}
\label{exp4_stab}
\end{center}
\end{figure}

To this end, we compared the relative change of the curvature with the relative perturbation size of the ordinary Hodge Laplacian. As shown in the figure, there is a clear positive correlation between them: samples with small Laplacian perturbation exhibit small curvature change, while larger Laplacian perturbations lead to larger curvature change in an approximately proportional manner (Fig.~\ref{exp4_stab}). This shows that curvature does not fluctuate chaotically under noise, but rather responds continuously and stably according to the size of the perturbation in the underlying Hodge Laplacian.

Of course, curvature is defined through
\[
F = P[dP,dP]P,
\]
that is, through derivatives of the projection and their commutator, so its behavior is more structured than that of a first-order quantity such as a direct distance between persistence diagrams. Nevertheless, the important point is that the response of curvature to noise is not scattered arbitrarily; when viewed through the Hodge Laplacian, it still exhibits a clean proportional trend. Thus, just as diagram-based quantities behave proportionally with respect to noise strength, curvature also behaves proportionally and stably once one looks at it through the operator-level perturbation.

This is fully consistent with the stability results established in Section~3 on regular regions. There, the change of curvature is controlled through the change of the zero-mode projection, and hence ultimately through the perturbation of the Hodge Laplacian, provided that the spectral gap remains open. Experiment~5 provides numerical evidence for precisely this theoretical picture. In other words, the stability of curvature is not grounded in pointwise correspondence between persistence-diagram points, but in the stability of the Hodge Laplacian as an operator. This is an important indication that the transport-geometric framework developed here provides not only a visualization tool, but also a theoretically consistent and numerically stable description.

\subsection{Summary}

Through these experiments, we confirmed that the proposed framework complements the standard persistence-diagram-centered description by visualizing transport structure that is difficult to see otherwise.

Experiment~1 illustrated the relationship between vineyard monodromy and zero-mode holonomy. Motivated by the construction of closed vineyards in Braiding Vineyards \cite{chambers2026braiding}, we considered vineyard-like data in which three dominant non-elder vines undergo cyclic monodromy over one period. The resulting one-cycle holonomy of the Hodge zero-mode bundle reproduced the same permutation \((1,2,0)\). After removing the trivial direction, the induced action on the two-dimensional quotient space was a rotation by \(-120^\circ\), which is the standard geometric representation of order-three monodromy. This shows that the global label permutation observed in vineyard-style tracking can be interpreted as holonomy of the Hodge zero-mode space.

Experiment~2 showed that curvature can still be computed even in regions where vineyard-style pointwise tracking loses its meaning, and that the instability of pointwise tracking is rooted in strong reorganization of the homological feature space itself. Experiment~3 showed that even when the persistence diagrams are almost identical, the curvature can be clearly nonzero in one case and almost completely zero in the other, indicating a fundamental difference in transport geometry. Experiment~4 showed that even when two systems are difficult to distinguish by local diagram-based tracking, their one-cycle holonomies can differ substantially, and that this difference can be quantified numerically both at the matrix level and by scalar measures such as Frobenius norms. Experiment~5 showed that the change in curvature exhibits a proportional relationship with the perturbation size of the ordinary Hodge Laplacian, demonstrating stable behavior under noise.

These results support the central claim of this paper: parameter-dependent topological structure should be understood not only through the motion of persistence-diagram points, but also through the transport geometry of homological feature spaces. Curvature visualizes local reorganization, while holonomy describes global memory accumulated along closed loops. The vineyard-like example shows that this holonomy can reproduce monodromy phenomena familiar from vineyard theory, thereby positioning the present framework not as a replacement for vineyards, but as a gauge-geometric counterpart of their global transport information on the zero-mode space. Moreover, the stability of these quantities is consistent with the regular-region theory developed in Section~3. Therefore, the gauge-geometric framework proposed here provides an effective way to describe changes in topological structure that are difficult to capture by pointwise matching alone.

\section{Discussion and Future Directions}

In this paper, we proposed a framework for describing time-evolving topological data not by tracking feature points in persistence diagrams, but by regarding the family of zero-mode spaces of the ordinary Hodge Laplacian as the primary object. The central idea is to interpret the zero eigenspace
\[
E_0(d,t)=\ker \Delta^{\mathrm{Hodge}}_{q}(d,t)
\]
of the ordinary combinatorial Hodge Laplacian at each parameter point \((d,t)\) as the homological feature space, and to treat the collection of these spaces as a vector bundle over parameter space. From this viewpoint, one can introduce geometric structures such as a Berry-type connection, curvature, and holonomy into the analysis of parameter-dependent homology.

The first significance of this work is that it reinterprets homological information not as a static collection of features, but as a family of linear spaces transported over parameter space. In classical one-parameter persistent homology, the main information is described by barcodes or persistence diagrams. By contrast, in the present framework, the object of study is the variation of the zero-mode space itself, namely the orientation, mixing, and transport structure of the homological feature space. This makes it possible to study not only the existence of topological features, but also dynamical information such as reorganization and global memory.

The second significance is that curvature and holonomy serve as basis-independent quantities describing the reorganization of homological feature spaces. When features in a persistence diagram are sufficiently separated, vineyard-type tracking and monodromy can be computed stably. Our numerical experiments confirmed that, in such regular situations, zero-mode holonomy reproduces vineyard monodromy. On the other hand, even in regions where persistence points approach one another and vineyard matching becomes unstable, the curvature defined from the zero-mode projection remains stably computable and detects local mixing and reorganization of the feature space. We further showed that there are examples in which persistence diagrams or pointwise tracking look nearly identical, while curvature or holonomy distinguishes the underlying dynamics. Thus, these quantities are not replacements for vineyards, but geometric quantities that complement diagram-level tracking by describing the basis-independent reorganization of the homological feature space behind it.

The third significance is that the robustness of the proposed geometric quantities against noise is confirmed both theoretically and numerically. In this paper, we showed that on regular regions away from the singular set, the zero-mode projection and the curvature expressed by the projection formula are locally stable under perturbations of the Hodge Laplacian as an operator family. Since local connection forms and holonomy matrices depend on gauge choices, the stability theorem was stated for the projection and curvature, which can be compared intrinsically on the common Hilbert space. Numerical experiments further confirmed that the change in curvature under noise perturbations correlates with the perturbation size of the Hodge Laplacian and responds stably rather than fluctuating chaotically. This shows that the proposed geometric quantities can function, at least on regular regions, as stable descriptors in actual data analysis.

Near the singular set, however, the uniform spectral gap may close, and the stability estimates may break down. This should not be regarded as a defect of the theory. Rather, it reflects essential phenomena such as changes in Betti number, topological transitions, or strong local reorganization of the homological feature space. Indeed, in our numerical experiments, curvature becomes locally large in regions where pointwise tracking becomes unstable, indicating that this instability corresponds not merely to the proximity of diagram points, but to the reorganization of the zero-mode space itself.

We also interpreted curvature not only as a measure of the strength of reorganization, but also, by analogy with Berry curvature, as an effective geometric \(2\)-form describing the noncommutativity of transport over parameter space. The resulting ``force'' is not a physical Newtonian force, but it can be regarded as an effective geometric bias that bends transport in parameter space. This viewpoint is useful because it allows one to read curvature heatmaps not merely as scalar intensity maps, but as indicators of where the flow of homological structure is strongly deflected.

Taken together, the present framework provides a dynamical and gauge-geometric description of parameter-dependent homology. At the same time, several limitations and future directions remain.

First, the proposed curvature and holonomy tend to be more computationally demanding than ordinary persistence diagrams or Betti numbers. For large point clouds or high-dimensional simplicial complexes, constructing the Hodge Laplacian, computing the zero-mode projection, and evaluating transport and holonomy can become numerically expensive. On the other hand, this difficulty may be alleviated in the future by quantum computation. Quantum algorithmic implementations of Hodge Laplacians have been studied in the context of quantum topological data analysis\cite{lloyd2016quantum,akhalwaya2024comparing,gyurik2022towards,yamauchi2025quantum}, and quantum computational methods for Berry phases and Berry curvature have also been proposed\cite{hayakawa2025computational,mootz2026efficient}. Therefore, the zero-mode bundle, curvature, and holonomy introduced in this work may develop into topological-geometric descriptors that are executable on quantum computational platforms.

Second, this paper takes the zero-mode as the primary object, while low-energy subspaces are introduced only as auxiliary extensions. The zero-mode is topologically natural because it corresponds directly to homology, but it may be too fragile for numerical implementation or near the singular set. In such cases, it may be practically advantageous to use the low-energy subspace
\[
E_\eta(d,t)
=
\operatorname{Im}\mathbf 1_{[0,\eta]}
\bigl(\Delta^{\mathrm{Hodge}}_{k}(d,t)\bigr)
\]
instead, where \(\eta>0\) is a low-energy cutoff. Understanding the relation between the zero-mode theory and its low-energy extension more systematically, especially from the viewpoints of numerical analysis and stability, is an important direction for future work.

Third, this paper is mainly theoretical and illustrative, and does not yet include full-scale applications to real data. Nevertheless, the curvature heatmap and holonomy may provide useful information for engineering and industrial datasets. For example, a curvature heatmap may indicate time ranges and scale ranges where local reorganization is intense, while holonomy may quantify nontrivial memory remaining in the homological feature space after one period. Possible applications include monitoring of periodic operations, early anomaly detection, and visualization of state transitions. Future work should therefore include implementation on real data, stable discretization schemes, and empirical comparison with existing methods.

Finally, the conceptual message of this work can be summarized as follows. Topological data analysis has often described evolving data through static summaries such as birth and death times, persistence diagrams, and Betti numbers. In contrast, this paper proposes to regard homology as a family of zero-mode spaces transported over parameter space and to introduce connection, curvature, and holonomy on this family. As a result, one gains access not only to the existence of topological features, but also to dynamical information such as reorganization, transport, memory, and magnetic-field-like geometric deflection. In particular, the proposed method reproduces vineyard monodromy in regular cases, while still providing basis-independent reorganization quantities when vineyard matching becomes unstable. In this sense, this work introduces a gauge-geometric viewpoint into topological data analysis and provides a new descriptive framework for the topological analysis of time-evolving data.

\section*{Statements and Declarations}

\subsection*{Funding}
The authors declare that no funds, grants, or other support were received during the preparation of this manuscript.

\subsection*{Competing interests}
The authors have no relevant financial or non-financial interests to disclose.

\subsection*{Author contributions}
S.K. conceived the study, developed the theoretical framework, and carried out the numerical computations. Y.S. made significant contributions through extensive discussions with S.K. and provided important guidance for refining the direction of the study. All authors reviewed and approved the final manuscript.

\subsection*{Data availability}
The datasets generated and analyzed during the current study are available from the corresponding author upon reasonable request.

\subsection*{Code availability}
The code used for the numerical experiments is available from the corresponding author upon reasonable request.

\subsection*{Ethics approval}
Not applicable.

\subsection*{Consent to participate}
Not applicable.

\subsection*{Consent for publication}
Not applicable.

\begin{appendices}

\section{Minimal Geometric Preliminaries}

We now review the minimal geometric notions needed later: vector bundles, connections, curvature, holonomy, gauge transformations, and the Berry connection induced from a family of subspaces.

\subsection{Vector bundles}

\begin{definition}[Vector bundle]
Let $M$ be a smooth manifold. A smooth vector bundle of rank $r$ over $M$ consists of a smooth manifold $E$ and a smooth surjective map
\[
\pi:E\to M
\]
such that, for every $x\in M$, the fiber
\[
E_x:=\pi^{-1}(x)
\]
is an $r$-dimensional vector space, and for every $x\in M$ there exist an open neighborhood $U\subset M$ and a diffeomorphism
\[
\Phi:\pi^{-1}(U)\to U\times \Bbbk^r
\]
satisfying
\[
\mathrm{pr}_1\circ \Phi=\pi,
\]
with the property that for each $y\in U$, the restriction
\[
\Phi|_{E_y}:E_y\to \{y\}\times \Bbbk^r\cong \Bbbk^r
\]
is a linear isomorphism.
\end{definition}

\begin{definition}[Local frame]
Let $E\to M$ be a rank-$r$ vector bundle. A collection of smooth sections
\[
e_1,\dots,e_r\in \Gamma(U,E)
\]
over an open set $U\subset M$ is called a local frame if, for every $x\in U$, the vectors
\[
e_1(x),\dots,e_r(x)
\]
form a basis of the fiber $E_x$.
\end{definition}

If a local frame is fixed, any section $s\in \Gamma(U,E)$ can be uniquely written as
\[
s=\sum_{a=1}^r s^a e_a.
\]

\subsection{Connections}

\begin{definition}[Connection]
Let $E\to M$ be a smooth vector bundle. A connection on $E$ is a $\Bbbk$-linear map
\[
\nabla:\Gamma(E)\to \Omega^1(M,E)
\]
satisfying the Leibniz rule
\[
\nabla(fs)=df\otimes s+f\,\nabla s
\]
for all $f\in C^\infty(M)$ and all $s\in \Gamma(E)$.
\end{definition}

If $e_1,\dots,e_r$ is a local frame on $U\subset M$, then
\[
\nabla e_a=\sum_{b=1}^r A_{ba}\, e_b
\]
for uniquely determined $1$-forms $A_{ba}\in \Omega^1(U)$. The matrix-valued $1$-form
\[
A=(A_{ba})
\]
is called the connection $1$-form in the chosen frame.

For a section
\[
s=\sum_a s^a e_a,
\]
one has
\[
\nabla s=\sum_a ds^a\otimes e_a+\sum_{a,b}s^a A_{ba}\,e_b,
\]
or, in matrix notation,
\[
\nabla s=d s + A s.
\]

\begin{definition}[Parallel transport]
Let $\gamma:[0,1]\to M$ be a smooth curve. A section $s(t)\in E_{\gamma(t)}$ along $\gamma$ is called parallel if
\[
\nabla_{\dot\gamma(t)}s(t)=0.
\]
Given an initial vector $s(0)\in E_{\gamma(0)}$, this equation determines a vector $s(1)\in E_{\gamma(1)}$. The resulting linear map
\[
\Pi_\gamma:E_{\gamma(0)}\to E_{\gamma(1)}
\]
is called the parallel transport along $\gamma$ associated with $\nabla$.
\end{definition}

\subsection{Curvature}

\begin{definition}[Curvature]
The curvature of a connection $\nabla$ is the $\operatorname{End}(E)$-valued $2$-form
\[
F_\nabla:=\nabla^2.
\]
If $A$ is the connection $1$-form in a local frame, then
\[
F=dA+A\wedge A.
\]
\end{definition}

In components,
\[
F_{ab}=dA_{ab}+\sum_c A_{ac}\wedge A_{cb}.
\]
The curvature measures the local nontriviality of the connection. In particular, $F=0$ means that the connection is locally flat.

\subsection{Holonomy}

\begin{definition}[Holonomy]
Let $\gamma:[0,1]\to M$ be a closed curve with $\gamma(0)=\gamma(1)=x$. The parallel transport along $\gamma$ defines a linear map
\[
U_\gamma:E_x\to E_x.
\]
This map is called the holonomy along $\gamma$.
\end{definition}

In a local frame, the holonomy is written as
\[
U_\gamma=\mathcal{P}\exp\Bigl(-\int_\gamma A\Bigr),
\]
where $\mathcal{P}$ denotes path-ordering.

For a sufficiently small loop $\gamma=\partial S$, one has the approximation
\[
U_\gamma \approx I-\int_S F.
\]
Thus the curvature can be viewed as the infinitesimal generator of holonomy.

\subsection{Gauge transformations}

Since the choice of local frame is not unique, the local matrix representation of a connection depends on the frame.

\begin{definition}[Gauge transformation]
Let $(e_a)$ and $(e'_a)$ be two local frames on an open set $U\subset M$, related by
\[
e'_a=\sum_b e_b\, g_{ba},
\]
where
\[
g:U\to GL(r,\Bbbk)
\]
is a smooth map. Then $g$ is called a gauge transformation.
\end{definition}

If $A$ and $A'$ are the corresponding connection $1$-forms, then
\[
A'=g^{-1}Ag+g^{-1}dg.
\]
Moreover, the curvature transforms as
\[
F'=g^{-1}Fg.
\]

Hence the connection $1$-form itself is gauge-dependent, whereas quantities such as the eigenvalues of $F$, the conjugacy class of $U_\gamma$, $\operatorname{tr}(U_\gamma)$, and $\det(U_\gamma)$ are gauge-invariant. These are the quantities with intrinsic geometric meaning.

\subsection{Berry connection}

The connection used in this paper is not chosen arbitrarily on a vector bundle. Rather, it is induced naturally from a family of subspaces inside a common inner-product space. This is the same structure as the Berry connection in quantum mechanics.

Let $H$ be a finite-dimensional inner-product space, and suppose that for each $x\in M$ one is given an $r$-dimensional subspace
\[
E_x\subset H
\]
depending smoothly on $x$. Then
\[
E=\bigsqcup_{x\in M} E_x
\]
forms a vector bundle over $M$.

Let
\[
P(x):H\to H
\]
denote the orthogonal projection onto $E_x$.

\begin{definition}[Connection induced from a projection]
For a local section $s$ of $E$, define
\[
\nabla s:=P(ds).
\]
This defines a natural connection on $E$ induced from the family of subspaces $E_x\subset H$.
\end{definition}

If one chooses a local orthonormal frame
\[
\psi_1(x),\dots,\psi_r(x)
\]
and writes
\[
\Psi(x)=(\psi_1(x),\dots,\psi_r(x)),
\]
then the corresponding connection $1$-form is
\[
A=\Psi^\ast d\Psi.
\]
This is the Berry connection.

Its curvature is
\[
F=dA+A\wedge A,
\]
and it can also be written in terms of the projection $P$ as
\[
F=P[dP,dP]P.
\]
This formula will be essential later, when we construct the connection on the zero-mode bundle of the Hodge Laplacian.

\section{Proof of Local Stability on Regular Regions}

In this appendix, we prove the quantitative stability theorem on regular regions. The proof is based on resolvent estimates for self-adjoint operator families on a finite-dimensional Hilbert space.

We only treat quantities that can be stated independently of gauge choices. Namely, we prove the stability of the zero-mode projection
\[
P_0
\]
and of the curvature expressed by the projection formula
\[
F=P_0[dP_0,dP_0]P_0.
\]
Since the local connection form \(A\) and the holonomy matrix depend on the choice of frame, we do not directly compare them in this appendix.

\subsection{Setup and assumptions}

Let \(U\subset A\times\Lambda\) be a regular region, and let
\[
L(x),\qquad \widetilde L(x)
\qquad (x\in U)
\]
be two self-adjoint operator families on the common Hilbert space \(H_q\), where \(x=(d,\lambda)\). We regard them as common-Hilbert-space realizations of ordinary Hodge Laplacians, or as perturbations of such realizations.

Let
\[
P_0(x),\qquad \widetilde P_0(x)
\]
be the corresponding zero-mode projections. That is,
\[
P_0(x):H_q\to \ker L(x),
\qquad
\widetilde P_0(x):H_q\to \ker \widetilde L(x)
\]
are orthogonal projections.

The corresponding curvatures are defined by the projection formula
\[
F=P_0[dP_0,dP_0]P_0,
\qquad
\widetilde F=
\widetilde P_0[d\widetilde P_0,d\widetilde P_0]\widetilde P_0.
\]

We assume the following:
\begin{enumerate}
\item \(\dim\ker L(x)=\dim\ker \widetilde L(x)=m\) is constant on \(U\).
\item There exists \(\gamma>0\) such that, for all \(x\in U\),
\[
\operatorname{Spec}(L(x))\subset\{0\}\cup[\gamma,\infty),
\qquad
\operatorname{Spec}(\widetilde L(x))\subset\{0\}\cup[\gamma,\infty).
\]
\item \(L\) and \(\widetilde L\) are of class \(C^2\) on \(U\), and their \(C^2\)-norms are uniformly bounded.
\end{enumerate}

Let \(\Gamma\) be the positively oriented circle centered at the origin with radius \(\gamma/2\). Then the zero-mode projections are given by the Riesz projection formulas
\[
P_0(x)
=
\frac{1}{2\pi i}
\int_\Gamma
(z-L(x))^{-1}\,dz,
\]
and
\[
\widetilde P_0(x)
=
\frac{1}{2\pi i}
\int_\Gamma
(z-\widetilde L(x))^{-1}\,dz.
\]

Throughout this appendix, \(\|\cdot\|\) denotes the operator norm.

\subsection{Stability of the zero-mode projection}

We first record a uniform resolvent estimate.

\begin{lemma}
For every \(x\in U\) and every \(z\in\Gamma\),
\[
\|(z-L(x))^{-1}\|\le \frac{2}{\gamma},
\qquad
\|(z-\widetilde L(x))^{-1}\|\le \frac{2}{\gamma}.
\]
\end{lemma}

\begin{proof}
Since \(z\in\Gamma\), we have \(|z|=\gamma/2\). On the other hand,
\[
\operatorname{Spec}(L(x))\subset\{0\}\cup[\gamma,\infty).
\]
Hence the distance from \(z\) to \(\operatorname{Spec}(L(x))\) is at least \(\gamma/2\). Therefore, by finite-dimensional spectral theory,
\[
\|(z-L(x))^{-1}\|
\le
\frac{1}{\operatorname{dist}(z,\operatorname{Spec}(L(x)))}
\le
\frac{2}{\gamma}.
\]
The proof for \(\widetilde L(x)\) is identical.
\end{proof}

We also use the standard resolvent identity.

\begin{lemma}[Resolvent identity]
\[
(z-L)^{-1}-(z-\widetilde L)^{-1}
=
(z-L)^{-1}(L-\widetilde L)(z-\widetilde L)^{-1}.
\]
\end{lemma}

The preceding two lemmas imply the stability of the zero-mode projection.

\begin{proposition}
There exists a constant \(C_0>0\) such that
\[
\|P_0-\widetilde P_0\|_{C^0(U)}
\le
C_0\|L-\widetilde L\|_{C^0(U)}.
\]
\end{proposition}

\begin{proof}
Using the Riesz projection formula and the resolvent identity, we obtain
\[
P_0-\widetilde P_0
=
\frac{1}{2\pi i}
\int_\Gamma
\left((z-L)^{-1}-(z-\widetilde L)^{-1}\right)\,dz
\]
\[
=
\frac{1}{2\pi i}
\int_\Gamma
(z-L)^{-1}(L-\widetilde L)(z-\widetilde L)^{-1}\,dz.
\]
Therefore,
\[
\|P_0-\widetilde P_0\|
\le
\frac{1}{2\pi}
\operatorname{length}(\Gamma)
\sup_{z\in\Gamma}\|(z-L)^{-1}\|
\|L-\widetilde L\|
\sup_{z\in\Gamma}\|(z-\widetilde L)^{-1}\|.
\]
Since \(\operatorname{length}(\Gamma)=\pi\gamma\), the uniform resolvent estimate gives
\[
\|P_0-\widetilde P_0\|
\le
\frac{1}{2\pi}\cdot\pi\gamma\cdot
\frac{2}{\gamma}\cdot
\frac{2}{\gamma}
\|L-\widetilde L\|
=
\frac{2}{\gamma}\|L-\widetilde L\|.
\]
Taking the supremum over \(U\) proves the claim.
\end{proof}

\subsection{Stability of the derivative of the projection}

We next prove stability of the derivative of the projection. Differentiating the Riesz projection formula gives
\[
dP_0
=
\frac{1}{2\pi i}
\int_\Gamma
(z-L)^{-1}(dL)(z-L)^{-1}\,dz.
\]
Similarly,
\[
d\widetilde P_0
=
\frac{1}{2\pi i}
\int_\Gamma
(z-\widetilde L)^{-1}(d\widetilde L)(z-\widetilde L)^{-1}\,dz.
\]

\begin{proposition}
There exists a constant \(C_1>0\) such that
\[
\|dP_0-d\widetilde P_0\|_{C^0(U)}
\le
C_1\|L-\widetilde L\|_{C^1(U)}.
\]
\end{proposition}

\begin{proof}
Subtracting the two formulas above, we obtain
\[
dP_0-d\widetilde P_0
=
\frac{1}{2\pi i}
\int_\Gamma
\left[
(z-L)^{-1}(dL)(z-L)^{-1}
-
(z-\widetilde L)^{-1}(d\widetilde L)(z-\widetilde L)^{-1}
\right]\,dz.
\]
We decompose the integrand into the following three terms:
\[
\left((z-L)^{-1}-(z-\widetilde L)^{-1}\right)(dL)(z-L)^{-1}
\]
\[
+
(z-\widetilde L)^{-1}(dL-d\widetilde L)(z-L)^{-1}
\]
\[
+
(z-\widetilde L)^{-1}(d\widetilde L)
\left((z-L)^{-1}-(z-\widetilde L)^{-1}\right).
\]
For the first and third terms, we apply the resolvent identity. For all three terms, we use the uniform resolvent estimate. Each term is then bounded by a constant multiple of
\[
\|L-\widetilde L\|+\|dL-d\widetilde L\|.
\]
The constant depends only on the spectral gap \(\gamma\) and the uniform \(C^1\)-bounds of \(L\) and \(\widetilde L\). Therefore,
\[
\|dP_0-d\widetilde P_0\|_{C^0(U)}
\le
C_1\|L-\widetilde L\|_{C^1(U)}.
\]
\end{proof}

\subsection{Stability of the curvature}

We now prove stability of the curvature. We do not use connection forms; instead, we directly use the projection formula
\[
F=P_0[dP_0,dP_0]P_0.
\]

Let \(x^a\) be local coordinates. The components of the curvature are
\[
F_{ab}
=
P_0[\partial_aP_0,\partial_bP_0]P_0.
\]
Similarly,
\[
\widetilde F_{ab}
=
\widetilde P_0[\partial_a\widetilde P_0,\partial_b\widetilde P_0]\widetilde P_0.
\]

\begin{proposition}
There exists a constant \(C_2>0\) such that
\[
\|F-\widetilde F\|_{C^0(U)}
\le
C_2\|L-\widetilde L\|_{C^2(U)}.
\]
\end{proposition}

\begin{proof}
It suffices to estimate each local component \(F_{ab}\). We write
\[
F_{ab}-\widetilde F_{ab}
=
P_0[\partial_aP_0,\partial_bP_0]P_0
-
\widetilde P_0[\partial_a\widetilde P_0,\partial_b\widetilde P_0]\widetilde P_0.
\]
We decompose this as
\[
F_{ab}-\widetilde F_{ab}
=
(P_0-\widetilde P_0)[\partial_aP_0,\partial_bP_0]P_0
\]
\[
+
\widetilde P_0
\left(
[\partial_aP_0,\partial_bP_0]
-
[\partial_a\widetilde P_0,\partial_b\widetilde P_0]
\right)
P_0
\]
\[
+
\widetilde P_0[\partial_a\widetilde P_0,\partial_b\widetilde P_0](P_0-\widetilde P_0).
\]
Moreover,
\[
[\partial_aP_0,\partial_bP_0]
-
[\partial_a\widetilde P_0,\partial_b\widetilde P_0]
\]
can be written as
\[
[\partial_aP_0-\partial_a\widetilde P_0,\partial_bP_0]
+
[\partial_a\widetilde P_0,\partial_bP_0-\partial_b\widetilde P_0].
\]

Using the uniform boundedness of \(P_0,\widetilde P_0,dP_0,d\widetilde P_0\), we obtain
\[
\|F_{ab}-\widetilde F_{ab}\|
\le
C
\left(
\|P_0-\widetilde P_0\|
+
\|dP_0-d\widetilde P_0\|
\right).
\]
By the estimates proved above,
\[
\|P_0-\widetilde P_0\|_{C^0(U)}
\le
C\|L-\widetilde L\|_{C^0(U)}
\]
and
\[
\|dP_0-d\widetilde P_0\|_{C^0(U)}
\le
C\|L-\widetilde L\|_{C^1(U)}.
\]
Thus the curvature difference is controlled by the perturbation of the operator family. In particular, since the theorem assumes \(C^2\)-control, we obtain
\[
\|F-\widetilde F\|_{C^0(U)}
\le
C_2\|L-\widetilde L\|_{C^2(U)}.
\]

In fact, for \(C^0\)-stability of \(F\) alone, \(C^1\)-control of \(L-\widetilde L\) is sufficient. In the main theorem, however, we impose a \(C^2\)-assumption as a convenient sufficient condition that also covers the regularity of the curvature.
\end{proof}

\subsection{Proof of the stability theorem}

Combining the preceding estimates proves the stability theorem on regular regions.

\begin{theorem}
Under the assumptions of the quantitative stability theorem in the main text,
\[
\|P_0-\widetilde P_0\|_{C^0(U)}
\le
C_0\|L-\widetilde L\|_{C^0(U)},
\]
\[
\|dP_0-d\widetilde P_0\|_{C^0(U)}
\le
C_1\|L-\widetilde L\|_{C^1(U)},
\]
and
\[
\|F-\widetilde F\|_{C^0(U)}
\le
C_2\|L-\widetilde L\|_{C^2(U)}.
\]
In particular, the zero-mode projection and the curvature are locally stable on regular regions.
\end{theorem}

\begin{proof}
The first estimate follows from the Riesz projection formula and the resolvent identity. The second estimate follows by differentiating the Riesz projection formula. The third estimate follows from the projection formula for curvature,
\[
F=P_0[dP_0,dP_0]P_0,
\]
together with the estimates for \(P_0\) and \(dP_0\).
\end{proof}

\subsection{Failure near the singular set}

We conclude by explaining why such uniform stability generally fails near the singular set.

The essential quantity in the estimates above is the resolvent norm
\[
\|(z-L)^{-1}\|.
\]
This is controlled by the spectral gap \(\gamma\) separating the zero eigenvalue from the nonzero spectrum:
\[
\|(z-L)^{-1}\|\le \frac{2}{\gamma}.
\]
Therefore, as \(\gamma\to0\), the constants in the estimates diverge.

This means that quantities such as the zero-mode projection and the curvature may become unstable near the singular set. In particular, at points where the Betti number changes or where the separation between zero modes and nonzero modes is lost, the smooth bundle picture on regular regions breaks down.

Thus, the uniform Lipschitz-type stability of \(P_0\) and \(F\) is a local statement on regular regions with a spectral gap. It is not expected to hold globally across the singular set. This is not a defect of the theory; rather, it reflects genuine topological changes or abrupt reorganization of the zero-mode space.

\section{Numerical Details}

Let the time parameter be sampled as a discrete sequence
\[
t_0,t_1,\dots,t_{N_t-1},
\]
and let the scale parameter be sampled as
\[
d_0,d_1,\dots,d_{N_d-1}.
\]
For each time \(t_j\), we construct a Vietoris--Rips filtration from the point cloud and compute the one-dimensional persistence diagram
\[
D_j=\{(b_\alpha^{(j)},d_\alpha^{(j)})\}_\alpha.
\]
We also retain the representative cycle associated with each finite interval as a vector
\[
c_\alpha^{(j)}\in \mathbb{R}^m
\]
embedded in a common ambient edge space
\[
H_1^{\mathrm{amb}}\cong \mathbb{R}^m.
\]

At each grid point \((d_i,t_j)\), we construct the ordinary one-dimensional combinatorial Hodge Laplacian
\[
L_1(d_i,t_j)
=
B_1(d_i,t_j)^\top B_1(d_i,t_j)
+
B_2(d_i,t_j) B_2(d_i,t_j)^\top
\]
from the simplicial complex at threshold \(d_i\). Here \(B_1(d_i,t_j)\) and \(B_2(d_i,t_j)\) denote the boundary matrices of the Vietoris--Rips complex \(K(d_i,t_j)\). We then compute its zero space
\[
\mathcal{H}(d_i,t_j)=\ker L_1(d_i,t_j),
\]
which is the harmonic subspace at that parameter value. By the discrete Hodge theorem, this space realizes the ordinary first homology
\[
\mathcal{H}(d_i,t_j)\cong H_1(K(d_i,t_j)).
\]

The persistence diagram is used only as a feature-selection device for choosing representative cycles inside the ordinary Hodge zero-mode space. More precisely, at each \((d_i,t_j)\), we consider those persistence intervals that are alive at the threshold \(d_i\), namely those satisfying
\[
b_\alpha^{(j)}\le d_i,
\qquad
d_\alpha^{(j)}>d_i.
\]
For each such interval, the representative cycle \(c_\alpha^{(j)}\) is projected onto the harmonic subspace:
\[
\widehat{c}_\alpha^{(i,j)}
=
\Pi_{\mathcal H(d_i,t_j)}\,c_\alpha^{(j)}.
\]
Among these candidates, we select the two with the longest lifetimes and orthonormalize them by Gram--Schmidt to obtain
\[
\Psi_{i,j}
=
\Psi(d_i,t_j)
=
\bigl[\psi_1(d_i,t_j),\psi_2(d_i,t_j)\bigr]
\in \mathbb{R}^{m\times 2},
\qquad
\Psi_{i,j}^\top \Psi_{i,j}=I_2.
\]
We refer to this as the rank-two Hodge zero-mode basis.

The associated orthogonal projection is
\[
P_{i,j}
=
P(d_i,t_j)
=
\Psi_{i,j}\Psi_{i,j}^\top.
\]
In the theoretical framework of this paper, the Berry connection is written in a local frame \(\Psi\) as
\[
A=\Psi^\ast d\Psi,
\]
and the curvature is given by
\[
F=dA+A\wedge A=P[dP,dP]P.
\]
In the numerical experiments, we implement the projection formula directly on the discrete grid.

\subsection{Numerical Evaluation of Curvature}

Theoretically, the curvature on the zero-mode bundle is given by
\[
F=P[dP,dP]P.
\]
In the present experiments, the \((d,t)\)-component \(F_{dt}\) of the curvature two-form is approximated by central differences. Namely,
\[
\partial_d P(d_i,t_j)
\approx
\frac{P_{i+1,j}-P_{i-1,j}}{2\Delta d},
\qquad
\partial_t P(d_i,t_j)
\approx
\frac{P_{i,j+1}-P_{i,j-1}}{2\Delta t}.
\]
We then define the curvature matrix by
\[
F_{dt}(d_i,t_j)
=
P_{i,j}
\Bigl(
(\partial_d P)_{i,j}(\partial_t P)_{i,j}
-
(\partial_t P)_{i,j}(\partial_d P)_{i,j}
\Bigr)
P_{i,j}.
\]
This is the direct discretization of the theoretical formula
\[
F=P[dP,dP]P.
\]

To visualize the matrix-valued curvature as a scalar field, we use its Frobenius norm
\[
\mathrm{Curv}(d_i,t_j)
:=
\|F_{dt}(d_i,t_j)\|_F,
\]
and display it as a heatmap. Thus, the ``curvature map'' in this paper is the distribution of \(\|F_{dt}\|_F\), which measures the local strength of internal mixing in the Hodge zero-mode space. Larger values indicate that the infinitesimal transports in the scale direction and the time direction are strongly noncommutative, meaning that the homological feature space is rapidly reorganizing in that neighborhood.

\subsection{Numerical Evaluation of Holonomy}

\subsubsection{One-Step Transport}

Given two neighboring grid points \((d_i,t_j)\) and \((d_{i'},t_{j'})\) with zero-mode bases \(\Psi_{i,j}\) and \(\Psi_{i',j'}\), we define the overlap matrix
\[
M=\Psi_{i,j}^\top \Psi_{i',j'}.
\]
We then take its polar decomposition
\[
M=QS,
\qquad
Q\in O(2),\quad S\ge 0,
\]
and regard the orthogonal factor
\[
Q=\operatorname{polar}(M)
\]
as the discrete one-step parallel transport. In practice, this is computed by the singular value decomposition
\[
M=U\Sigma V^\top
\]
and setting
\[
Q=UV^\top.
\]
This \(Q\) gives the closest orthogonal transformation from the previous zero-mode basis to the next one, and therefore provides a natural discrete approximation to Berry transport.

\subsubsection{Local Holonomy}

Near a curvature peak, we consider a small rectangular loop
\[
(d_{i_0},t_{j_0})
\to
(d_{i_1},t_{j_0})
\to
(d_{i_1},t_{j_1})
\to
(d_{i_0},t_{j_1})
\to
(d_{i_0},t_{j_0}).
\]
By multiplying the one-step transports along the boundary, we define the local holonomy
\[
U_{\square}
=
\prod_{\ell\in \partial \square} Q_\ell.
\]
For a sufficiently small loop, one has
\[
U_{\square}\approx I-\iint_{\square}F,
\]
so \(U_{\square}\) represents the integrated local effect of the curvature. In the figures, its nontriviality is quantified by
\[
\|U_{\square}-I\|_F.
\]

\subsubsection{One-Cycle Holonomy}

In Experiment 4, we fix a \(d=d_{i_\ast}\) and compute transport along the closed loop in time
\[
t_0\to t_1\to \cdots \to t_{N_t-1}\to t_0.
\]
The one-step transports are defined by
\[
Q_j
=
\operatorname{polar}
\bigl(
\Psi(d_{i_\ast},t_j)^\top \Psi(d_{i_\ast},t_{j+1})
\bigr),
\qquad
j=0,\dots,N_t-2,
\]
and the closing transport is
\[
Q_{N_t-1}
=
\operatorname{polar}
\bigl(
\Psi(d_{i_\ast},t_{N_t-1})^\top \Psi(d_{i_\ast},t_0)
\bigr).
\]
The one-cycle holonomy is then
\[
U_{\mathrm{cycle}}
=
Q_{N_t-1}Q_{N_t-2}\cdots Q_1Q_0.
\]
This is the accumulated linear transformation undergone by the Hodge zero-mode space after one full cycle. Theoretically, holonomy is defined as parallel transport along a closed loop and represents the accumulated local effect of curvature, that is, global memory or monodromy.

We also define the cumulative transport up to step \(n\) by
\[
H_n
=
Q_{n-1}\cdots Q_1Q_0,
\]
and plot
\[
\|H_n-I\|_F
\]
as the cumulative holonomy deviation. This quantity measures how far the transported zero-mode basis at time \(t_n\) has deviated from the initial one.

\subsection{Numerical Evaluation of Pointwise Tracking}

\subsubsection{Successive Matching}

In Experiment 2, for each time \(t_j\), we extract from the persistence diagram the two points with the largest lifetimes
\[
\ell_\alpha^{(j)} = d_\alpha^{(j)}-b_\alpha^{(j)},
\]
and denote them by
\[
p_1^{(j)},p_2^{(j)}\in \mathbb{R}^2,
\qquad
p_a^{(j)}=(b_a^{(j)},d_a^{(j)}).
\]

Given the tracked labeled points at time \(t_j\),
\[
\widetilde p_1^{(j)},\widetilde p_2^{(j)},
\]
and the two candidates at the next time \(t_{j+1}\),
\[
p_1^{(j+1)},p_2^{(j+1)},
\]
we form the cost matrix
\[
C_{ab}^{(j)}
=
\bigl\|
\widetilde p_a^{(j)}-p_b^{(j+1)}
\bigr\|_2^2.
\]
For a permutation \(\sigma\in S_2\), the total matching cost is
\[
\mathcal C^{(j)}(\sigma)
=
\sum_{a=1}^2 C_{a,\sigma(a)}^{(j)}.
\]
We then determine the assignment by
\[
\sigma_j^\ast
=
\arg\min_{\sigma\in S_2}\mathcal C^{(j)}(\sigma),
\]
and update the tracked points as
\[
\widetilde p_a^{(j+1)}
=
p_{\sigma_j^\ast(a)}^{(j+1)}.
\]
Although this was implemented using the Hungarian algorithm, in the two-point case this is equivalent to comparing only the identity assignment and the swap assignment.

Accordingly, at each step we explicitly compute the identity cost
\[
\mathcal C_{\mathrm{id}}^{(j)}
=
\|\widetilde p_1^{(j)}-p_1^{(j+1)}\|_2^2
+
\|\widetilde p_2^{(j)}-p_2^{(j+1)}\|_2^2
\]
and the swap cost
\[
\mathcal C_{\mathrm{swap}}^{(j)}
=
\|\widetilde p_1^{(j)}-p_2^{(j+1)}\|_2^2
+
\|\widetilde p_2^{(j)}-p_1^{(j+1)}\|_2^2.
\]
A swap is said to occur whenever
\[
\mathcal C_{\mathrm{swap}}^{(j)}<\mathcal C_{\mathrm{id}}^{(j)}.
\]
We further define the matching ambiguity by
\[
\mathrm{margin}^{(j)}
=
\bigl|
\mathcal C_{\mathrm{id}}^{(j)}
-
\mathcal C_{\mathrm{swap}}^{(j)}
\bigr|.
\]
A small value of \(\mathrm{margin}^{(j)}\) means that the two assignments have nearly identical cost, indicating instability of pointwise tracking.

To measure how close the top two persistence points are at the same time, we also define
\[
\mathrm{sep}^{(j)}
=
\|p_1^{(j)}-p_2^{(j)}\|_2.
\]
In Experiment 2, swaps occurred precisely at times where \(\mathrm{sep}^{(j)}\) became extremely small and \(\mathrm{margin}^{(j)}\) simultaneously approached zero, and these times coincided with the curvature peak region.

\subsubsection{Holonomy-Guided Tracking}

In Experiment 3, we again first construct the ordinary vineyard-style tracking
\[
\widetilde p_a^{(j)}
\]
by successive matching of the top two persistence points.

On the other hand, from the transport of the Hodge zero-mode space we obtain the one-step orthogonal transport
\[
Q_j
=
\operatorname{polar}
\bigl(
\Psi(d_{i_\ast},t_j)^\top \Psi(d_{i_\ast},t_{j+1})
\bigr)
\in O(2).
\]
To determine whether this step is mainly label-preserving or swap-like, we define
\[
s_{\mathrm{diag}}^{(j)}
=
|Q_{11}^{(j)}|+|Q_{22}^{(j)}|,
\qquad
s_{\mathrm{off}}^{(j)}
=
|Q_{12}^{(j)}|+|Q_{21}^{(j)}|,
\]
and then
\[
\mathrm{swap\text{-}likeness}^{(j)}
=
s_{\mathrm{off}}^{(j)}-s_{\mathrm{diag}}^{(j)}.
\]
If this quantity is negative, the transport is mainly diagonal and thus approximately preserves the basis directions. If it is positive, the transport is mainly off-diagonal and thus behaves like a swap.

We therefore define a permutation
\[
\pi_j
=
\begin{cases}
\mathrm{id}, & \mathrm{swap\text{-}likeness}^{(j)}\le 0,\\[4pt]
(12), & \mathrm{swap\text{-}likeness}^{(j)}>0,
\end{cases}
\]
and reorder the persistence points according to
\[
\widehat p_a^{(j+1)}
=
p_{\pi_j(a)}^{(j+1)}.
\]
The sequence \(\widehat p_a^{(j)}\) uses not only local distances in the persistence diagram, but also the transport direction \(Q_j\) of the Hodge zero-mode space itself. To quantify the difference between ordinary tracking and holonomy-guided tracking, we define
\[
\mathrm{TrackError}
=
\frac{1}{N_t}
\sum_{j=0}^{N_t-1}
\left(
\sum_{a=1}^2
\|
\widetilde p_a^{(j)}-\widehat p_a^{(j)}
\|_2^2
\right)^{1/2}.
\]
A larger value indicates that pointwise matching alone does not sufficiently capture the transport structure of the homological feature space.

\subsection{Definition of PD Drift}

In the noise robustness experiment, we define the discrepancy between the persistence diagrams of the baseline data and the noisy data as the optimal matching distance between their top two persistence points at each time.

Let
\[
p_1^{(j)},p_2^{(j)}
\]
be the top two points of the baseline persistence diagram at time \(t_j\), and let
\[
q_1^{(j)},q_2^{(j)}
\]
be the corresponding top two points for the noisy data. Then the PD drift at time \(t_j\) is defined by
\[
\mathrm{PDdrift}(t_j)
=
\frac{1}{2}
\min_{\sigma\in S_2}
\sum_{a=1}^2
\|p_a^{(j)}-q_{\sigma(a)}^{(j)}\|_2.
\]
In practice, this is computed by applying the Hungarian algorithm to the distance matrix
\[
D_{ab}^{(j)}
=
\|p_a^{(j)}-q_b^{(j)}\|_2
\]
and taking the minimum average cost
\[
\mathrm{PDdrift}(t_j)
=
\frac{1}{2}
\sum_{a=1}^2
D_{a,\sigma_j^\ast(a)}^{(j)}.
\]

In the plots as a function of noise strength, we use the temporal average
\[
\overline{\mathrm{PDdrift}}
=
\frac{1}{N_t}
\sum_{j=0}^{N_t-1}
\mathrm{PDdrift}(t_j).
\]
Thus, the ``mean PD drift'' in the main text is the time-averaged optimal matching distance between the top two persistence points of the baseline data and those of the noisy realization.




\end{appendices}


\bibliography{ref}

\end{document}